\documentclass[11pt]{amsart}
\usepackage{amssymb,amsmath,amsthm,newlfont,stmaryrd}
\newtheorem{prop}{Proposition}[section]
\newtheorem{thm}[prop]{Theorem}
\newtheorem{lem}[prop]{Lemma}
\newtheorem{cor}[prop]{Corollary}

\theoremstyle{remark}
\newtheorem{rem}[prop]{Remark}

\begin{document}

	\title[On the automatic regularity of derivations]
	{On the automatic regularity of derivations \\ from Riesz subalgebras of $\cal L^r(X)$}
	
	\author{A.~Blanco}
	\address{Mathematical Sciences Research Centre, Queen's University Belfast}
	
	\email{a.blanco@qub.ac.uk}
	
	
	
	\begin{abstract}
		We investigate the automatic regularity of bounded derivations from a Banach lattice algebra of regular operators $\cal A$ into a Banach $\cal A$-module with a Banach lattice structure compatible with the module operations. 
	\end{abstract}
	
	\maketitle

	\section{Introduction}
	
	The problem of finding conditions on a pair of Banach lattices $X$, $Y$, under which every continuous linear map from $X$ to $Y$ is regular (i.e., a linear combination of positive linear maps) is an old one. For instance, it is shown in work of Kantorovich and Vulikh \cite{KV}, dating back to the 1930's, that if $X$ is an AL-space and $Y$ a KB-space, or if $Y$ is order complete and has a strong order unit, then every bounded linear map from $X$ to $Y$ is automatically regular. 
	
	In the setting of Banach lattice algebras, i.e., Banach algebras with a Banach lattice structure compatible with that of the algebra in the sense that the positive cone is closed under the algebra product, it seems only natural to consider the question of automatic regularity for bounded linear maps that somehow take into account the algebra structure. With this in mind, in \cite{Bl1} and \cite{Bl2}, we considered the question for bounded homomorphisms between Banach lattice algebras of regular operators. For instance, we showed in \cite{Bl1} that if the atoms in $X$ form a subsymmetric basis and $Y$ is atomic and reflexive, then every algebra homomorphism $\Theta:\cal A^r(X)\to\cal A^r(Y)$ (resp.~every injective algebra homomorphism $\Theta:\cal L^r(X)\to\cal L^r(Y)$ satisfying $\Theta(\cal L^r(X))\cap \cal F(Y)\ne\{0\}$) is necessarily regular. These results were then extended to the non-atomic setting in \cite{Bl2}. In the case of commutative Banach lattice algebras, it is known that if the algebras are semisimple, then any algebra homomorphism between them is in fact a Riesz homomorphism. This last result follows readily from more general results about semiprime Archimedean $f$-algebras, obtained in \cite{HdeP} (see for instance, Theorem~5.1).
	
	Much in the spirit of the above research, in this note we consider the question of automatic regularity for another very important class of linear maps one can associate to any Banach algebra, and that take into account the algebra structure, the so-called derivations. Recall a derivation from an algebra $\cal A$ to an $\cal A$-module $\frak X$ is a linear map $D:\cal A\to\frak X$ that satisfies the identity $D(ab) = aD(b) + D(a)b$ $(a,b\in\cal A)$. If $\cal A$ is a Banach lattice algebra, one can consider derivations from $\cal A$ to Banach $\cal A$-modules with a vector lattice structure compatible with that of the module, in the sense that the module products preserve positivity (we shall call such modules Banach lattice modules -- see next section). The main concern of the note will be precisely the automatic regularity of such maps. Like in \cite{Bl1} and \cite{Bl2}, the focus will be on the case where $\cal A$ is a Banach lattice algebra of regular operators, and conditions on $\cal A$, that guarantee the regularity of derivations, will be sought (like in \cite{Bl1} and \cite{Bl2}) in terms of the underlying Banach lattices.  
		
	Apart from being of independent interest, there is yet one other reason for our concern with the question of automatic regularity of derivations, and is the close relation between derivations and algebra homomorphisms. It is well-known that if $\cal A$ is an algebra, $\frak X$ is an $\cal A$-module and $D:\cal A\to\frak X$ is a derivation, then $\cal B := \cal A\oplus\frak X$, with multiplication defined by $(a,x)(b,y) := (ab,ay+xb)$ $((a,x),(b,y)\in\cal B)$, is an algebra and $\Theta:\cal A\to\cal B$, $a\mapsto (a,D(a))$ is an algebra homomorphism. If $\cal A$ is a Riesz algebra and $\frak X$ is a Riesz space and an $\cal A$-module such that $ax,xa\in\frak X_+$ whenever $a\in\cal A_+$ and $x\in\frak X_+$, then $\cal B$ is also a Riesz algebra for the order induced by the cone $\cal A_+\times\frak X_+$. Moreover, if $\cal A$ is a Banach lattice algebra and $\frak X$ is a Banach lattice $\cal A$-module (see next section) then $\cal B$, endowed with the norm $\|(a,x)\| := \|a\|+\|x\|$ $((a,x)\in\cal B)$, is also a Banach lattice algebra.
	
	It is now fairly easy to see that if $\cal A$ is a Riesz algebra, $\frak X$ is a Riesz $\cal A$-module and $D:\cal A\to\frak X$ is a derivation, then $D$ is regular if and only if $\Theta:\cal A\to\cal B$ (with $\cal B$ and $\Theta$ as above) is regular. One could thus use the results from \cite{Bl1} and \cite{Bl2} to obtain results on automatic regularity of derivations from Banach lattice algebras of regular operators. However, some of the hypotheses of the main results of \cite{Bl1} and \cite{Bl2} were tailored to the situation in which the codomain of the homomorphism is a Riesz algebra ideal of the algebra of regular operators on some Banach lattice. For this reason, a direct approach to the problem seems more appropriate. Apart from simplifying our arguments and making the note more self-contained, it will yield more general results. Some of the assumptions of \cite[Theorem~3.2]{Bl1} will be retained, though, like for instance, the existence of a sequence of o-minimal pairwise orthogonal idempotents in the domains of our maps, being after all a fairly common feature of the Banach lattice algebras we are interested in. 
	
	It is the author's hope, in view of the connection explained above, that the results of the note will in turn shed further light into the more general problem of identifying the exact conditions under which every bounded algebra homomorphism from a Banach lattice algebra of regular operators to another Banach lattice algebra is regular. It might be worth pointing out, that when $\frak X\ne\{0\}$ the algebra $\cal A\oplus\frak X$, as defined above, is  not semiprime (and hence, semisimple). As a result, in the case where $\cal A$ is commutative and $D:\cal A\to\frak X$ is a derivation it does not seem possible to draw (at least directly) from the results of \cite{HdeP}, any conclusion about the regularity of the algebra homomorphism $\Theta$ associated with $D$ (as above), and hence, about the regularity of $D$. 
	
	We should add that a more general class of modules over Banach lattice algebras than the one we consider here, was introduced in \cite{Ro}, under the name of regular Banach lattice modules (mainly with the purpose of defining a theory of regular cohomology). In the case of a regular $\cal A$-module $\frak X$, the algebra homomorphisms from $\cal A$ to the algebra of bounded operators on $\frak X$, induced by the module products, are only assumed to be regular -- here, we consider only the special case in which they are positive. Although we have not checked all the details, it seems the results/arguments of the note can be extended with little difficulty to the more general class of regular Banach lattice modules. One such extension could prove useful in comparing some of the standard (bounded) cohomological notions with their order counterparts (see \cite{Ro}).   
	
	Regarding the organization of the note, in the next section the reader will find most of the notation and terminology used in the paper. Section~3 is devoted to the atomic case. In it, we shall first prove a general result about continuous derivations from Riesz topological algebras to Riesz topological modules (see below for definitions); then proceed, using this result, to establish our main result concerning derivations from Banach lattice algebras of regular operators on atomic Banach lattices. Some concrete cases are considered at the end of the section. Lastly, in Section~4, we present one possible extension of the main result of Section~3 to the situation in which the underlying Banach lattices of the algebras under consideration are not purely atomic. We conclude the section (and the note) by looking at some concrete situations in which the hypotheses of this last result hold.

	\section{Some background and terminology}
	
	In this section, we have gathered some background and terminology used throughout the note.
	
	Given a normed space $X$, we shall denote by $X'$ its topological dual. Furthermore, if $E$ is a subset of $X$, we shall write $\mathrm{sp}(E)$ for its linear span, $[E]$ for the norm closure of $\mathrm{sp}(E)$ and $E^\perp$ for the subspace $\{x'\in X' : E\subseteq\ker{x'}\}$ of $X'$. We shall denote by $X_{[\rho]}$ the closed ball centered at the origin and of radius $\rho$.
	
	By a Riesz space, we shall mean a vector lattice. We shall denote its positive cone by $X_+$. Recall also that a complex Riesz space $X$ is defined as the complexification $X_\Bbb R + iX_\Bbb R$ of a real Riesz space $X_\Bbb R$ in which $\sup_{\theta\in [0,2\pi)} \big((\cos{\theta})\xi + (\sin{\theta})\eta\big)$ exists for any pair $\xi,\eta\in X_{\Bbb R}$. The modulus of an element $x\in X$ is then defined by the formula $|x| := \sup_{\theta\in [0,2\pi)} \big((\cos{\theta})\mathrm{Re}(x) + (\sin{\theta})\mathrm{Im}(x)\big)$, where $\mathrm{Re}:X\to X_{\Bbb R}$, $\xi+i\eta\mapsto\xi$, and $\mathrm{Im}:X\to X_{\Bbb R}$, $\xi+i\eta\mapsto\eta$. A complex normed Riesz space $X$ is then -- by definition -- the complexification of a real normed Riesz space $(X_{\Bbb R},\|\cdot\|_{\Bbb R})$, endowed with the norm $\|x\| := \||x|\|_{\Bbb R}$ $(x\in X)$. 
	
	By a {\bf Riesz algebra} we shall mean an associative algebra $\cal A$, which is also a Riesz space, where $ab\geq 0$ $(a,b\in\cal A_+)$, or equivalently, where $|ab|\leq |a||b|$ $(a,b\in\cal A)$. As in \cite{Bl1}, by a {\bf Riesz algebra ideal} of a Riesz algebra $\cal A$, we shall mean an algebra ideal of $\cal A$, which is also a Riesz subspace. Elements $a$ and $b$ of a Riesz algebra shall be called {\bf orthogonal} if $ab = 0 = ba$ and {\bf disjoint} if $|a|\vee |b| = 0$. Like in \cite{Bl1} and \cite{Bl2}, by an {\bf o-minimal idempotent} of a Riesz algebra $\cal A$, we shall mean an idempotent $e\in\cal A_+$ such that $e\cal Ae = \Bbb Ke$, where $\Bbb K$ stands for the underlying field. We shall say $\cal A$ is semiprime if $\{0\}$ is the only two-sided ideal of $\cal A$ such that $\cal I^2 = \{0\}$.
	
	By a {\bf topological algebra} we shall mean an associative algebra $\cal A$ with a topology $\tau$, such that $(\cal A,\tau)$ is a Hausdorff locally convex topological vector space and multiplication on $\cal A$ is separately continuous. By an approximate identity (a.i.~in short) for a topological algebra $\cal A$, we shall mean a net $(e_i)$ in $\cal A$ such that $\lim_i e_ia = a = \lim_i ae_i$ $(a\in\cal A)$. Also, following \cite{Bl1}, we shall call a sequence $(b_n)$ in a topological algebra $\cal A$, {\bf convergence preserving} if whenever a sequence $(a_n)$ in $\cal A$ converges, so do the sequences $(a_nb_n)$ and $(b_na_n)$. Furthermore, we shall say that a sequence $(a_n)$ in $\cal A$, {\bf factors (or that there is a factorization of it) through} another sequence $(b_n)$ in $\cal A$ if there are convergence preserving sequences $(u_n)$ and $(v_n)$ in $\cal A_+$ such that $a_n = u_n b_{k_n}v_n$ $(n\in\Bbb N)$ for some strictly increasing sequence $(k_n)\subset\Bbb N$. (Note that, unlike in \cite{Bl1}, here we do not require $(u_n)$ and $(v_n)$ to be bounded.) A Riesz algebra $\cal A$, endowed with a norm $\|\cdot\|$, with respect to which $\cal A$ is both a Banach lattice and a Banach algebra shall be said to be a {\bf Banach lattice~algebra}.
	
	Given a Riesz algebra $\cal A$, by a {\bf Riesz $\cal A$-module} we shall mean an $\cal A$-module $\frak X$, which is also a Riesz space, in which $ax,xa\in\frak X_+$ whenever $a\in\cal A_+$ and $x\in\frak X_+$. This last is in fact equivalent to $|ax|\leq |a||x|$ and $|xa|\leq |x||a|$ $(a\in\cal A,\,x\in\frak X)$ (this is straightforward in the real case; for an argument that covers both the real and complex cases see at the end of this section). Given an $\cal A$-module $\frak X$, we shall write $\cal A\frak X$ (resp.~$\frak X\cal A$) for the set $\{ax : a\in\cal A,\, x\in\frak X\}$ (resp.~$\{xa : a\in\cal A,\, x\in\frak X\}$). By a {\bf topological $\cal A$-module}, we shall mean an $\cal A$-module $\frak X$, endowed with a topology $\omega$, with respect to which $(\frak X,\omega)$ is a Hausdorff locally convex topological vector space and both module products, $(a,x)\mapsto ax$, $\cal A\times\frak X\to\frak X$, and $(x,a)\mapsto xa$, $\frak X\times\cal A\to\frak X$, are separately continuous. If $\cal A$ is a Banach lattice algebra and $\frak X$ is a Riesz $\cal A$-module, endowed with a norm $\|\cdot\|$ such that $(\frak X,\|\cdot\|)$ is both a Banach lattice and a Banach $\cal A$-module (i.e., there is some constant $M$ such that $\|ax\|\leq M\|a\|\|x\|$ and $\|xa\|\leq M\|a\|\|x\|$ $(a\in\cal A,\,x\in\frak X)$), then $\frak X$ shall be said to be a {\bf Banach lattice $\cal A$-module}.
	
	Given a linear map $T:X\to Y$ we shall denote by $T(X)$ or $\mathrm{im}\,T$ its image. By a {\bf positive} linear map $T:X\to Y$ between Riesz spaces $X$ and $Y$, we shall mean a linear map such that $T(X_+)\subseteq Y_+$. A linear map between Riesz spaces shall be said to be {\bf regular} if it is a linear combination of positive linear maps. Given Riesz spaces $X$ and $Y$, the collection of all regular maps between them, endowed with the order induced by the cone of positive operators, forms an ordered vector space denoted $\cal L^r(X,Y)$. It is well-known that for $T\in\cal L^r(X,Y)$, $|T|$ exists if and only if $\sup_{|\xi|\leq x} |T\xi|$ exists for every $x\in X_+$. Moreover, if this last holds then $|T|(x) = \sup_{|\xi|\leq x} |T\xi|$ $(x\in X_+)$. In particular, $\cal L^r(X,Y)$ is a Riesz space whenever $Y$ is order complete. In the complex case, $\cal L^r(X,Y)$ is the vector space complexification of $\{T\in\cal L^r(X,Y) : T(X_{\Bbb R})\subseteq Y_{\Bbb R}\}$. 
	
	When $X$ and $Y$ are Banach lattices, $\cal L^r(X,Y)$ is a subspace of the Banach space $\cal B(X,Y)$ of all bounded linear operators from $X$ to $Y$. Endowed with the norm $\|T\|_r := \inf\{\|S\| : S\in\cal B(X,Y),\,|Tx|\leq S|x|\; (x\in X)\}$ $(T\in\cal L^r(X,Y))$, $\cal L^r(X,Y)$ becomes a Banach space, and if $Y$ is order complete, a Banach lattice. To simplify notations, we continue to denote by $\cal L^r(X,Y)$ the Banach space $(\cal L^r(X,Y),\|\cdot\|_r)$. When $X = Y$, we write $\cal L^r(X)$ for $\cal L^r(X,X)$. With composite of maps as product, the latter becomes a Banach algebra, and of course, a Banach lattice algebra if $X$ is order complete. We write $\cal F(X)$ for the linear space of all continuous finite-rank operators on $X$ and $\cal A^r(X)$ for its norm-closure in $\cal L^r(X)$. The latter is well-known to be always a Banach lattice (see \cite[Theorem~1]{Schw}), and if $X$ is atomic, also an order ideal of $\cal L^r(X)$. Note that $\cal A^r(X)$ is a Banach lattice algebra. If $X$ is complex then $\cal A^r(X)$ is the complexification of~$\cal A^r(X_{\Bbb R})$.
	
	Recall also that if $X$ is a Riesz space and $P:X\to X$ is a positive projection, then $P(X)$ is a lattice subspace of $X$, i.e., endowed with the order induced by the positive cone $P(X)\cap X_+$, $P(X)$ is a vector lattice. Its modulus function is given by $x\mapsto P|x|$ $(x\in P(X))$. Moreover, in the case where $X$ is a Banach lattice, $x\mapsto \|P|x|\|$ $(x\in P(x))$ is a Banach lattice norm on $P(X)$, equivalent to the original norm. (For a proof of these facts, see \cite[Theorem~5.59]{AA}). Lastly, by a {\bf band or order projection} on a Banach lattice $X$, we shall mean a projection $P\in\cal L^r(X)$ such that $0\leq P\leq \mathrm{id}_X$. If $P\in\cal L^r(X)$ is a band projection then $|Px| = P|x|$ $(x\in X)$. 
	
	We end the section with a proof of the fact that if $\cal A$ is a Riesz algebra and $\frak X$ is a Riesz $\cal A$-module, then $|ax|\leq |a||x|$ and $|xa|\leq |x||a|$ $(a\in\cal A,\,x\in\frak X)$ (once again, this is mainly for the sake of the complex case, which is not as straightforward as the real case, and for which we do not have a reference). To this end let $\frak Y$ be an order complete Riesz space containing $\frak X$ as a Riesz subspace. Then $\phi:\cal A\to\cal L^r(\frak X,\frak Y)$, $a\mapsto [\phi(a):\frak X\to\frak Y,\; x\mapsto ax]$, is positive and $|\phi(a)|\leq \phi(|a|)$ $(a\in\cal A)$. It follows from the formula for the module of an   element in $\cal L^r(\frak X,\frak Y)$ (see above) that for every $a\in\cal A$ and $x\in\frak X$, $|ax|\leq |\phi(a)|(|x|)\leq \phi(|a|)(|x|) = |a||x|$. A similar argument shows that $|xa|\leq |x||a|$ $(a\in\cal A,\,x\in\frak X)$.
	
	Further terminology will be introduced throughout the note where is needed. For any unexplained material on Banach lattices or regular operators we refer the reader to \cite{AB} and \cite{Sc}.

	\section{The atomic case}
	
	Our first result is a characterization of regularity of derivations in a special situation that involves no topological assumptions. To some extent it will serve as motivation for subsequent results.
	
	\begin{thm}\label{LA}
		Let $\cal A$ be a Riesz algebra with identity $e$ and a set $\{p_i : i\in I\}$ of pair-wise orthogonal o-minimal idempotents such that $\sup\{\sum_{i\in F} p_i : F\subseteq I \text{ finite}\} = e$ and $\dim{p_i\cal Ap_j}\leq 1$ $(i,j\in I)$; let $\frak X$ be an order complete Riesz $\cal A$-module such that $\sup\{(\sum_{i\in F} p_i)x(\sum_{i\in G} p_i) : F,G\subseteq I \text{ finite}\} = exe$ $(x\in\frak X_+)$; and let $D:\cal A\to\frak X$ be a derivation. Then $D$ is regular if and only if $\{\sum_{i\in F} |D(p_i)| : F\subseteq I \text{ finite}\}$ is order bounded.
	\end{thm}
	
	\begin{proof}
		To simplify notations, let $\cal F$ stand for the collection of all finite subsets of $I$. Suppose first $\{\sum_{i\in F} |D(p_i)| : F\in\cal F\}$ is order-bounded and let $\xi := \sup\{\sum_{i\in F} |D(p_i)| : F\in\cal F\}$. Since $D(t) = D(e)t + eD(t)e + tD(e)$ $(t\in\cal A)$, it will suffice to show that $\{|eD(t)e| : |t|\leq a\}$ is order bounded for every $a\in\cal A_+$.
		
		Set $\pi_F := \sum_{i\in F} p_i$ $(F\in\cal F)$. Then, for every $t\in\cal A$ and $F,G\in\cal F$,
		\[
		\begin{split}
			\pi_F D(\pi_F&t\pi_G)\pi_G \\
			&= \pi_F\bigg(\sum_{i\in F} \sum_{j\in G} D(p_itp_j)\bigg)\pi_G \\
			&= \pi_F \bigg(\sum_{i\in F} \sum_{j\in G}
			\big(D(p_i)p_itp_j + p_iD(p_itp_j)p_j + p_itp_jD(p_j)\big)\bigg)\pi_G \\
			&=\pi_F \sum_{i\in F} D(p_i)p_it\pi_G +
			\sum_{i\in F} \sum_{j\in G} p_iD(p_itp_j)p_j +
			\pi_Ft\sum_{j\in G} p_jD(p_j)\pi_G,
		\end{split}
		\]
		and so,
		\begin{equation}\label{F}
			\begin{split}
				|\pi_FD(\pi_F&t\pi_G)\pi_G|
				\leq e\xi|t| +
			    \sum_{i\in F} \sum_{j\in G} |p_iD(p_itp_j)p_j| +
				|t|\xi e.
			\end{split}
		\end{equation}
		
		Since $\dim{p_i\cal Ap_j}\leq 1$ $(i,j\in I)$, for every pair $i,j\in I$ there exists $e_{ij}\in\cal A_+$ such that for every $t\in\cal A$, $p_itp_j = \lambda_t e_{ij}$ for some $\lambda_t$ in the underlying field. If $|t|\leq a\in\cal A_+$, then $|\lambda_t| e_{ij} = |p_itp_j|\leq p_iap_j = \lambda_a e_{ij}$, and hence,
		\[
		|p_iD(p_itp_j)p_j| = |\lambda_t||p_iD(e_{ij})p_j|\leq \lambda_a |p_iD(e_{ij})p_j| = |p_iD(p_iap_j)p_j|.
		\]
		Thus, if $|t|\leq a\in\cal A_+$, 
		\[
		\begin{split}
			\sum_{i\in F} \sum_{j\in G} |p_iD(p_it&p_j)p_j|
			\leq \sum_{i\in F} \sum_{j\in G} |p_iD(p_iap_j)p_j| \\
			&\leq \sum_{i\in F} \sum_{j\in G} |p_iD(p_i)ap_j|
			+ \sum_{i\in F} \sum_{j\in G} |p_iD(a)p_j| +
			\sum_{i\in F} \sum_{j\in G} |p_iaD(p_j)p_j| \\
			&\leq \pi_F\bigg(\sum_{i\in F} |D(p_i)|\bigg)a\pi_G
			+ \pi_F|D(a)|\pi_G +
			\pi_F a\bigg(\sum_{j\in G} |D(p_j)|\bigg)\pi_G \\
			&\leq e\xi a + e|D(a)|e + a\xi e.
		\end{split}
		\]
		Combining (\ref{F}) with the latter estimate, we obtain that
		\begin{equation}\label{0}
			\sup\{|\pi_F D(\pi_F t\pi_G)\pi_G| : |t|\leq a
			\text{ and } F,G\in\cal F\}\leq 2e\xi a + e|D(a)|e + 2a\xi e.
		\end{equation}
		
		Next, as $e\geq 0$, the map $x\mapsto exe$, $\frak X\to e\frak Xe$, is a positive projection and $e\frak Xe$ is, accordingly, a lattice subspace of $\frak X$. If we let $\leq_e$ stand for the order induced in $e\frak Xe$ by the positive cone $e\frak Xe\cap\frak X_+$, then the module of an element $z\in (e\frak Xe,\leq_e)$, which we shall denote $|z|_e$, is given by the formula $|z|_e = e|z|e$. Clearly, each map $y\mapsto\pi_F y$, $e\frak X e\to e\frak X e$, and $y\mapsto y\pi_G$, $e\frak X e\to e\frak X e$, is a band projection on $(e\frak Xe,\leq_e)$ and we thus have that
		\[
		|\pi_F D(t)\pi_G|_e = |\pi_F eD(t)e\pi_G|_e = \pi_F |eD(t)e|_e\pi_G,
		\]
		for every $t\in\cal A$ and $F,G\in\cal F$. It follows from this and our hypotheses that
		\begin{equation}\label{T}
		    \begin{split}
				|eD(t)e|_e = e|eD(t)e|_e e &= \sup_{F,G} \pi_F|eD(t)e|_e\pi_G \\
				&= \sup_{F,G} |\pi_F D(t)\pi_G|_e = \sup_{F,G} e|\pi_F D(t)\pi_G|e \quad (t\in\cal A).
			\end{split}
		\end{equation}
		Now,
		\begin{equation}\label{Fo}
			\pi_F D(t)\pi_G = D(\pi_F t\pi_G) - D(\pi_F)t\pi_G - \pi_F tD(\pi_G),
		\end{equation}
		and so, combining (\ref{T}), (\ref{Fo}) and (\ref{0}), we finally obtain that
		\[
		\begin{split}
			|eD(t)e|\leq |eD(t)e|_e &\leq
			\sup_{F,G} e|\pi_F D(\pi_F t\pi_G)\pi_G|e \\
			&\hspace{2cm} + \sup_{F,G} e|\pi_F D(\pi_F)t\pi_G|e
			+ \sup_{F,G} e|\pi_F tD(\pi_G)\pi_G|e \\
			&\leq 3e\xi a + e|D(a)|e + 3a\xi e,
		\end{split}
		\]
		whenever $|t|\leq a$, as required. As explained earlier, the regularity of $D$ follows from this. 
		
		Conversely, suppose $D$ is regular. The maps $x\mapsto p_ix$ and $x\mapsto xp_i$ $(i\in I)$ are all band projections on $(e\frak Xe,\leq_e)$. Accordingly, since $D$ is regular, for every $F\in\cal F$ one has that,
		\[
		\begin{split}
			\sum_{i\in F} |D(p_i)| &= \sum_{i\in F} |D(ep_ip_ie)| \\
			&\leq \sum_{i\in F} |D(e)p_i| + \sum_{i\in F} |eD(p_i)p_i|
			+ \sum_{i\in F} |p_iD(p_i)e| + \sum_{i\in F} |p_iD(e)| \\
			&\leq |D(e)|e + \sum_{i\in F} |eD(p_i)p_i|_e
			+ \sum_{i\in F} |p_iD(p_i)e|_e + e|D(e)| \\
			&\leq |D(e)|e + \bigvee_{i\in F} |eD(p_i)p_i|_e
			+ \bigvee_{i\in F} |p_iD(p_i)e|_e + e|D(e)| \\
			&\leq |D(e)|e + 2e|D|(e)e + e|D(e)|.
		\end{split}
		\]
	\end{proof}
	
	We should point out that the hypotheses of Theorem~\ref{LA} are satisfied by many natural examples of Riesz algebras. Simply note that if $X$ is an atomic Riesz space and $\{p_i : i\in I\}$ is the family of all principal band projections generated by atoms in $X$, then any subalgebra of $\cal L^r(X)$, which is also a Riesz subspace and contains $\{\mathrm{id}_X\}\cup \{p_i : i\in I\}$, satisfies the conditions of the theorem. 
	
	On the other hand, any quotient of the form $\cal L^r(X)/\cal A^r(X)$, where $X$ is an infinite dimensional atomic Banach lattice, provides an example of a Riesz algebra (and also, of a Riesz $\cal L^r(X)$-module) for which the hypotheses of the theorem clearly fail. We do not know whether the situation of such examples constitutes a real obstruction to the conclusion of Theorem~\ref{LA}, or is simply an indication of the limitations of our methods.  
		
	Next, we turn our attention to the central problem of the note, i.e., that of finding conditions on a Banach lattice algebra of regular operators $\cal A$ and on a Banach lattice $\cal A$-module $\frak Y$, forcing every bounded derivation from $\cal A$ to $\frak Y$ to be regular. Key to our main result in this direction will be the next theorem, very much inspired by the one we just proved, and similar in spirit to \cite[Theorem~3.2]{Bl1}. In it, we shall make use of topological assumptions to replace the collection of o-minimal idempotents of Theorem~\ref{LA}, by a suitable subcollection. Following \cite{Sc}, we shall call a non-empty subset $S$ of a Riesz space $X$, {\bf absolutely majorized}, if the set $\{|x| : x\in S\}$ is bounded above. (Of course, if $X$ is real, $S$ is absolutely majorized if and only if it is order-bounded.)   
	
	\begin{thm}\label{32like}
		Let $(\cal A,\tau)$ be a semiprime Riesz and topological algebra, let $(\frak Y,\omega)$ be an order complete Riesz and topological $\cal A$-module, and let $D:\cal A\to\frak Y$ be a sequentially continuous derivation. Suppose
		\begin{itemize}
			\item[--] $\cal A_+$ is closed and contains an a.i.~$(e_n)$ for $\cal A$;
			
			\item[--] there is a locally convex Hausdorff topology $\widetilde{\omega}$ on $\frak Y$ (in the complex case, so that $x+iy\mapsto x-iy$, $\frak Y_{\Bbb R}+i\frak Y_{\Bbb R}\to\frak Y_{\Bbb R}+i\frak Y_{\Bbb R}$, is continuous) such that $\frak Y_+$ is $\widetilde{\omega}$-closed, absolutely majorized subsets of $\frak Y$ are $\widetilde{\omega}$-precompact, for every $a\in\cal A$ the maps $y\mapsto ay$ and $y\mapsto ya$ from $\frak Y$ to $\frak Y$ are $\widetilde{\omega}$-$\widetilde{\omega}$-continuous, and $\widetilde{\omega}$ is finer than $\omega$ on the $\widetilde{\omega}$-closures of $\cal A\frak Y$ and $\frak Y\cal A$;
			
			\item[--] there is in $\cal A$ a sequence $(p_i)$ of pairwise orthogonal o-minimal idempotents such that:
			\begin{itemize}
				\item[i)] the sequence $\pi_n := \sum_{i=1}^n p_i$ $(n\in\Bbb N)$ is convergence preserving;
				
				\item[ii)] there is a factorization $e_n = u_n\pi_{k_n}v_n$ $(n\in\Bbb N)$ of $(e_n)$ through $(\pi_n)$, such that $\{D(u_n) : n\in\Bbb N\}$, $\{D(v_n) : n\in\Bbb N\}$, $\{u_ny : n\in\Bbb N\}$ and $\{yv_n : n\in\Bbb N\}$ $(y\in\frak Y)$ are $\widetilde{\omega}$-precompact; and
				
				\item[iii)] $\sup_n \sum_{i=1}^n |p_iD(p_i)\pi_n|$ and $\sup_n \sum_{i=1}^n |\pi_nD(p_i)p_i|$ exist.
			\end{itemize}
		\end{itemize}
		Then $D$ is regular.
	\end{thm}
	
	In proving Theorem~\ref{32like}, we shall make use of the following stronger version of the lemma contained in the proof of \cite[Theorem~3.2]{Bl1}).
	
	\begin{lem}\label{aux}
		Let $(x_n)$ and $(y_n)$ be sequences in $\cal A$ such that, for every convergent sequence $(z_n)$ in $\cal A$, $(x_nz_n)$ and $(z_ny_n)$ converge. Let $(a_{mn})$ be a double sequence in $\cal A$ such that $\lim_m \lim_n a_{mn} = a$ and suppose $\lim_m \lim_n x_ma_{mn}y_n = \eta$. Then $\eta = \lim_m \lim_n x_may_n$.
	\end{lem}
	
	\begin{proof}
		First note that if $(z_n)$ is a sequence in $\cal A$ converging to 0, then $(x_nz_n)$ and $(z_ny_n)$ must converge to 0 as well. To see this, let $\{\|\cdot\|_i : i\in I\}$ be a system of seminorms inducing the topology on $\cal A$. Suppose $\lim_n x_nz_n = z\ne 0$ and let $\mu\in I$ be such that $\|z\|_\mu\ne 0$. Since $\lim_n \|z_n\|_\mu = 0$, there is a sequence $(t_n)$ in $\Bbb R_+$ such that $t_n\uparrow\infty$ and $\lim_n \|t_nz_n\|_\mu = 0$. But then we would have that $(t_nx_nz_n)$ converges, while on the other hand, $\lim_n |\|t_nx_nz_n\|_\mu - t_n\|z\|_\mu| = 0$, which is clearly impossible. Thus $z = 0$. The proof that $(z_ny_n)$ converges to 0 is completely analogous.
		
		Now, set $a_m := \lim_n a_{mn}$ $(m,n\in\Bbb N)$. It follows readily from the previous paragraph that $\lim_m \lim_n x_ma_{mn}y_n = \lim_m x_m \lim_n a_{mn}y_n = \lim_m x_m \lim_n a_my_n$. We will show next that $\lim_m x_m \lim_n a_my_n = \lim_m x_m \lim_n ay_n$, or equivalently, that $\lim_m x_m \lim_n (a_m-a)y_n = 0$. Set $z_n := a_n-a$ $(n\in\Bbb N)$ and suppose towards a contradiction that $\lim_m x_m \lim_n z_my_n\ne 0$. The latter means that for some $\alpha\in I$, there exists a subsequence $(x_{m_k} \lim_n z_{m_k}y_n)$ such that $\inf_k \|x_{m_k} \lim_n z_{m_k}y_n\|_\alpha>0$. By our assumptions on $(x_m)$, the sequence $(\lim_n z_{m_k}y_n)$ cannot converge to 0 either, and so, for some $\beta\in I$ there is a subsequence $(\lim_n z_{m_{k_l}}y_n)$ such that $\inf_l \|\lim_n z_{m_{k_l}}y_n\|_\beta =: \delta>0$. From this, one can inductively define a subsequence $(y_{n_l})$ of $(y_n)$ such that $\|z_{m_{k_l}}y_{n_l}\|_\beta\geq \delta/2$. On the one hand, this would mean that $(z_{m_{k_l}}y_{n_l})$ cannot converge to 0, while on the other hand, the fact that multiplication from the right by $(y_{n_l})$ preserves convergence, combined with the fact that $(z_{m_{k_l}})$ converges to 0, implies $\lim_l z_{m_{k_l}}y_{n_l} = 0$ (by the result from the previous paragraph). We thus have a contradiction, and hence, the desired conclusion.
	\end{proof}
	
	\begin{proof}[Proof of Theorem~\ref{32like}]
		Let $(u_n)$ and $(v_n)$ be convergence preserving sequences in $\cal A_+$ such that $e_n = u_n\pi_{k_n}v_n$ $(n\in\Bbb N)$ for some strictly increasing sequence $(k_n)$. Suppose, in addition, $(u_n)$ and $(v_n)$ have been chosen so that each one of the sets $\{D(u_n) : n\in\Bbb N\}$, $\{D(v_n) : n\in\Bbb N\}$, $\{u_ny : n\in\Bbb N\}$ and $\{yv_n : n\in\Bbb N\}$ $(y\in\frak Y)$ is $\widetilde{\omega}$-precompact. Set
		\[
		\Phi(a) := \lim_i \lim_j v_iau_j
		\quad\text{and}\quad
		\Psi(a) := \lim_r \lim_s \pi_{k_r}(v_rau_s)\pi_{k_s}
		\quad (a\in\cal A).
		\]
		Then
		\[
		\begin{split}
			D(a) &= D\Big(\lim_m \lim_n e_mae_n\Big) \\
			&= D\Big(\lim_m \lim_n u_m(\pi_{k_m}(v_mau_n)\pi_{k_n})v_n\Big)
			= \lim_m \lim_n D(u_m\Psi(a)v_n) \\
			&= \lim_m \lim_n \Big(D(u_m)\Psi(a)v_n +
			u_mD(\Psi(a))v_n + u_m\Psi(a)D(v_n)\Big) \\
			&= \lim_m \lim_n \Big(D(u_m)\Psi(a)v_n +
			u_m\Big(\lim_r \lim_s D(\pi_{k_r}\Phi(a)\pi_{k_s})\Big)v_n
			+ u_m\Psi(a)D(v_n)\Big)
		\end{split}
		\]
		(the third and last equalities by Lemma~\ref{aux}). 
		
		Let $\widetilde{x}$ and $\widetilde{y}$ be $\widetilde{\omega}$-limit points of $\{D(u_m) : m\in\Bbb N\}$ and $\{D(v_n) : n\in\Bbb N\}$, respectively, and let $(u_{m_\alpha})$ and $(v_{n_\beta})$ be subnets such that $\widetilde{\omega}$-$\lim_\alpha D(u_{m_\alpha}) = \widetilde{x}$ and $\widetilde{\omega}$-$\lim_\beta D(v_{n_\beta}) = \widetilde{y}$. Then
		\begin{equation}\label{summand1}
			\lim_\alpha \lim_n D(u_{m_\alpha})\Psi(a)v_{n}
			= \lim_\alpha D(u_{m_\alpha}) \lim_n \Psi(a)v_{n} = \widetilde{x} \lim_n \Psi(a)v_{n},
		\end{equation}
		(the second equality, because $y\mapsto y(\lim_n \Psi(a)v_{n})$, $\frak Y\to\frak Y$, is $\widetilde{\omega}$-$\widetilde{\omega}$-continuous -- by hypothesis) and similarly,
		\begin{equation}\label{summand2}
			\lim_m \lim_\beta u_{m}\Psi(a)D(v_{n_\beta})
			= \Big(\lim_m u_{m}\Psi(a)\Big) \widetilde{y}.
		\end{equation}
		
		We can further assume, by passing to subnets if needed, that $\lim_\alpha u_{m_\alpha}y$ and $\lim_\beta yv_{n_\beta}$ exist for every $(y\in\frak Y)$ (simply recall that $\{u_my : n\in\Bbb N\}$ and $\{yv_n : n\in\Bbb N\}$ are $\widetilde{\omega}$-precompact $(y\in\frak Y)$; a standard application of Tychonoff Theorem yields the existence of subnets of $(u_n)$ and $(v_n)$, as required). Let
		\[
		\Theta:\cal A\to\frak Y,\;
		t\mapsto \lim_\alpha \lim_\beta
		u_{m_\alpha}\Big(\lim_r \lim_s D(\pi_{k_r}t\pi_{k_s})\Big)v_{n_\beta}.
		\]
		Clearly, $\Theta$ is well-defined. We will show it is regular. To this end, let
		\[
		\xi := \sup_n \sum_{i=1}^n |\pi_nD(p_i)p_i|
		\quad\text{and}\quad
		\eta := \sup_n \sum_{i=1}^n |p_iD(p_i)\pi_n|.
		\]
		Let $a\in\cal A_+$ and let $t\in\cal A$ be such that $|t|\leq a$. Then, given $m,n,r,s\in\Bbb N$, with $r\geq m$ and $s\geq n$, one has that
		\begin{equation}\label{firststep}
			\begin{split}
				|u_m D(\pi_{k_r}t\pi_{k_s})v_n|
				&\leq u_m(\xi+\eta)av_n + u_m|D(Q(a))|v_n + u_ma(\xi+\eta)v_n.
			\end{split}
		\end{equation}
		Verification of (\ref{firststep}) is easy but is not short, so we shall postpone it until the end of the proof and proceed with the rest of the argument.
		
		Fix $m,n\in\Bbb N$, and to simplify notations, set $y_{rs} := D(\pi_{k_r}t\pi_{k_s})$ $(r,s\in\Bbb N)$ and
		\[
		\xi_{mn}(a) := u_m(\xi+\eta)av_n + u_m|D(Q(a))|v_n + u_ma(\xi+\eta)v_n.
		\]
		It follows from (\ref{firststep}) that each set $\{u_my_{rs}v_n : s\in\Bbb N\}$ $(m,n,r\in\Bbb N,\, r\geq m)$ is absolutely majorized, and so it has a compact $\widetilde{\omega}$-closure on $\frak Y$, which we shall denote $K_{mnr}$. If $K := \prod_{m,n,r\in\Bbb N,\,r\geq m} K_{mnr}$ with the product topology, then $K$ is compact and the sequence $y_s := (u_m y_{rs}v_n)_{m,n,r\in\Bbb N,\, r\geq m}$ $(s\in\Bbb N)$ must contain a subnet $(y_{s_\gamma})$ such that $(u_m y_{rs_\gamma}v_n)_\gamma$ $\widetilde{\omega}$-converges for every triple $m,n,r\in\Bbb N$ with $r\geq m$. For each pair $m,n\in\Bbb N$, let $\widetilde{y}_{mnr} := \widetilde{\omega}$-$\lim_\gamma u_m y_{rs_\gamma}v_n$ if $r\geq m$, and let $\widetilde{y}_{mnr} := 0\in\frak Y$ if $r<m$. As $\frak Y_+$ is $\widetilde{\omega}$-closed, it follows from (\ref{firststep}) that $|\widetilde{y}_{mnr}|\leq \xi_{mn}(a)$ $(m,n,r\in\Bbb N)$. Thus, $\{\widetilde{y}_{mnr} : r\in\Bbb N\}$ is absolutely majorized $(m,n\in\Bbb N)$, and so, it has a compact $\widetilde{\omega}$-closure, $K_{mn}$ say. The same argument as before, this time applied to the sequence $z_r := (\widetilde{y}_{mnr})_{m,n}$ $(r\in\Bbb N)$ as a subset of $\prod_{m,n} K_{mn}$ with the product topology, yields a subnet $(z_{r_\delta})$ such that, for every pair $m,n\in\Bbb N$, $(\widetilde{y}_{mnr_\delta})_\delta$ $\widetilde{\omega}$-converges to some $\widetilde{y}_{mn}\in K_{mn}$. One then has that $y_{mn} = \widetilde{\omega}\text{-}\lim_\delta \widetilde{y}_{mnr_\delta} = \widetilde{\omega}\text{-}\lim_\delta \widetilde{\omega}\text{-}\lim_\gamma u_my_{r_\delta s_\gamma}v_n$ $(m,n\in\Bbb N)$. Since $\lim_r \lim_s u_m y_{rs}v_n$ exists $(m,n\in\Bbb N)$ and $\widetilde{\omega}$ is finer than $\omega$ on the $\widetilde{\omega}$-closure of $\cal A\frak Y\cal A$,
		\begin{equation}\label{firstfact}
			\begin{split}
				\lim_r \lim_s u_m y_{rs}v_n &= \lim_\delta \lim_\gamma u_m y_{r_\delta s_\gamma}v_n = \widetilde{\omega}\text{-}\lim_\delta \widetilde{\omega}\text{-}\lim_\gamma u_m y_{r_\delta s_\gamma}v_n
				\quad (m,n\in\Bbb N).
			\end{split}
		\end{equation}
		Also
		\begin{equation}\label{secondfact}
			\lim_\alpha \lim_\beta u_{m_\alpha}yv_{n_\beta} = \widetilde{\omega}\text{-}\lim_\alpha \widetilde{\omega}\text{-}\lim_\beta u_{m_\alpha}yv_{n_\beta}\quad (y\in\frak Y)
		\end{equation}
		(the limit on the right exists by our choice of $(u_{m_\alpha})$ and $(v_{n_\beta})$, while the equality is again consequence of the fact that $\widetilde{\omega}$ is finer than $\omega$ on the $\widetilde{\omega}$-closure of $\cal A\frak Y\cal A$). Combining (\ref{firstfact}) and (\ref{secondfact}), we obtain that
		\[
		\begin{split}
			\Theta(t) &= \lim_\alpha \lim_\beta u_{m_\alpha}\Big(\lim_r \lim_s y_{rs}\Big)v_{n_\beta}
			= \widetilde{\omega}\text{-}\lim_\alpha \widetilde{\omega}\text{-}\lim_\beta u_{m_\alpha}\Big(\lim_r \lim_s y_{rs}\Big)v_{n_\beta} \\
			&= \widetilde{\omega}\text{-}\lim_\alpha \widetilde{\omega}\text{-}\lim_\beta \Big(\lim_r \lim_s u_{m_\alpha}y_{rs}v_{n_\beta}\Big)
			= \widetilde{\omega}\text{-}\lim_\alpha \widetilde{\omega}\text{-}\lim_\beta
			\Big(\widetilde{\omega}\text{-}\lim_\delta
			\widetilde{\omega}\text{-}\lim_\gamma u_{m_\alpha}y_{r_\delta s_\gamma}v_{n_\beta}\Big).
		\end{split}
		\]
		By (\ref{firststep}) and the closedness of $\frak Y_+$, $|\widetilde{\omega}\text{-}\lim_\delta \widetilde{\omega}\text{-}\lim_\gamma u_my_{r_\delta s_\gamma}v_n|\leq \xi_{m,n}(a)$ $(m,n\in\Bbb N)$, and so,
		\[
		|\Theta(t)|\leq \widetilde{\omega}\text{-}\lim_\alpha \widetilde{\omega}\text{-}\lim_\beta \xi_{m_\alpha n_\beta}(a),
		\]
		where the double limit exists by our choice of $(u_{m_\alpha})$ and $(v_{n_\beta})$. Since $a\in\cal A_+$ was arbitrary and $\frak Y$ is Dedekind complete, $\Theta$ is regular.
		
		Lastly, combining the expression for $D(a)$ from the beginning of the proof with (\ref{summand1}), (\ref{summand2}) and the definition of $\Theta$, one obtains that
		\begin{equation}\label{Formula}
			D(a) = \widetilde{x}\lim_n \Psi(a)v_n + (\Theta\circ\Phi)(a)
			+ \Big(\lim_m u_m\Psi(a)\Big)\widetilde{y} \quad (a\in\cal A).
		\end{equation}
		Clearly, $\Phi$ and $\Psi$ are positive. Furthermore, since $\cal A_+$ is closed, the maps $a\mapsto\lim_m u_ma$, $\cal A\to\cal A$, and $a\mapsto\lim_n av_n$, $\cal A\to\cal A$, are positive too. Taking into account these facts, together with the regularity of $\Theta$, one concludes that $D$ is regular.
		
		To finish, it only remains to prove (\ref{firststep}).
		
		Since $\cal A$ is semiprime, $\dim{p_i\cal Ap_j}\leq 1$ $(i,j\in\Bbb N)$ (see \cite[Lemma~3.1]{Bl1}), and the same argument as in the proof of Proposition~\ref{LA}, yields that if $|t|\leq a\in\cal A_+$ then $|p_iD(p_itp_j)p_j|\leq |p_iD(p_iap_j)p_j|$. Thus, if $|t|\leq a\in\cal A_+$ and $Q(a) := \lim_k \lim_l \pi_ka\pi_l$,
		\begin{equation}\label{ineq1}
			\begin{split}
				|p_iD(p_itp_j)p_j|&\leq |p_iD(p_iap_j)p_j| = |p_iD(p_iQ(a)p_j)p_j| \\
				&\leq |p_iD(p_i)Q(a)p_j| + |p_iD(Q(a))p_j| + |p_iQ(a)D(p_j)p_j|.
			\end{split}
		\end{equation}
		Note that $p_iD(p_i)Q(a)p_j = \lim_k p_iD(p_i)\pi_kap_j$ and $p_iQ(a)D(p_j)p_j = \lim_l p_ia\pi_lD(p_j)p_j$ $(a\in\cal A)$. Let $(\varepsilon_{ij})_{1\leq i\leq m,\, 1\leq j\leq n}$ be an arbitrary choice of scalars such that $|\varepsilon_{ij}| = 1$ $(1\leq i\leq m,\,1\leq j\leq n)$.
		Then, for $k\geq m$,
		\[
		\begin{split}
			\Bigg|\sum_{i=1}^m \sum_{j=1}^n \varepsilon_{ij} p_iD(p_i)\pi_kap_j\Bigg|
			&\leq \sum_{i=1}^m \sum_{j=1}^n |p_iD(p_i)\pi_k|ap_j \\
			&= \Big(\sum_{i=1}^m |p_iD(p_i)\pi_k|\Big)a\pi_n\leq \eta a\pi_n.
		\end{split}
		\]
		Let $\zeta_k := \sum_{i=1}^m \sum_{j=1}^n \varepsilon_{ij} p_iD(p_i)\pi_kap_j$ $(k\in\Bbb N)$. By the above, $(\zeta_k)$ is absolutely majorized, and so, it must contain a $\widetilde{\omega}$-convergent subnet, $(\zeta_{k_\alpha})$ say. Let $\widetilde{\zeta} := \widetilde{\omega}$-$\lim_\alpha \zeta_{k_\alpha}$. Note $\widetilde{\zeta}\in\frak Y\cal A$, for $\zeta_{k_\alpha} = \zeta_{k_\alpha}\pi_n$ for every $\alpha$, and $\widetilde{\omega}$-$\lim_\alpha \zeta_{k_\alpha}\pi_n = \widetilde{\zeta}\pi_n$. Since $\widetilde{\omega}$ is finer than $\omega$ on $\frak Y\cal A$ and $\frak Y_+$ is $\widetilde{\omega}$-closed, letting $k\to\infty$ on $\big|\sum_{i=1}^m \sum_{j=1}^n \varepsilon_{ij} p_iD(p_i)\pi_kap_j\big|\leq \eta a\pi_n$, one finds that
		\[
		\begin{split}
			\sum_{i=1}^m \sum_{j=1}^n \varepsilon_{ij} p_iD(p_i)Q(a)p_j
			&= \sum_{i=1}^m \sum_{j=1}^n \varepsilon_{ij} \;\omega\text{-}\lim_k p_iD(p_i)\pi_kap_j \\
			&= \omega\text{-}\lim_k \zeta_k = \widetilde{\omega}\text{-}\lim_\alpha \zeta_{k_\alpha}
			\leq \eta a\pi_n.
		\end{split}
		\]
		This last holds for any choice of scalars $(\varepsilon_{ij})$ of modulus 1, so we must have 
		\begin{equation}\label{est1}
			\sum_{i=1}^m \sum_{j=1}^n |p_iD(p_i)Q(a)p_j|\leq
			\eta a\pi_n
		\end{equation}
		(here we have used the fact that for any set of vectors $x_1,\ldots,x_n$ in a Riesz space $X$, $\sum_{i=1}^n |x_i| = \sup\{|\sum_{i=1}^n \lambda_i x_i| : |\lambda_i| = 1 \; (1\leq i\leq n)\}$, which can be easily obtained by induction from the known identity $|x|+|y| = \sup_{|\lambda| = |\mu| = 1} |\lambda x + \mu y|$ $(x,y\in X)$). Similarly, one shows that
		\begin{equation}\label{est2}
			\sum_{i=1}^m \sum_{j=1}^n |p_iQ(a)D(p_j)p_j|\leq
			\pi_m a\xi.
		\end{equation}
			
		Next note that for all $t\in\cal A$ and $m,n\in\Bbb N$,
		\[
		\begin{split}\label{id1}
			|\pi_mD(&\pi_mt\pi_n)\pi_n| \\
			&\leq \Bigg|\sum_{i=1}^m \pi_mD(p_i)p_i\Bigg||t|\pi_n +
			\sum_{i=1}^m \sum_{j=1}^n |p_iD(p_itp_j)p_j| +
			\pi_m|t|\Bigg|\sum_{j=1}^n p_jD(p_j)\pi_n\Bigg|
		\end{split}
		\]
		(this follows readily from the identity in the second paragraph of the proof of Proposition~\ref{LA}, with $F = \{1,\ldots,m\}$ and $G = \{1,\ldots,n\}$). Combining this estimate with (\ref{ineq1}), (\ref{est1}) and (\ref{est2}), one obtains that, for every $t\in\cal A$ with $|t|\leq a$,
		\[
		\begin{split}
			|\pi_mD(&\pi_mt\pi_n)\pi_n| \\
			&\leq \sum_{i=1}^m |\pi_mD(p_i)p_i|a\pi_n
			+ \bigg(\sum_{i=1}^m \sum_{j=1}^n |p_iD(p_i)Q(a)p_j|
			+ \sum_{i=1}^m \sum_{j=1}^n |p_iD(Q(a))p_j| \\
			&\hspace{2cm} + \sum_{i=1}^m \sum_{j=1}^n |p_iQ(a)D(p_j)p_j|\bigg)
			+ \pi_ma\sum_{j=1}^n |p_jD(p_j)\pi_n| \\
			&\leq (\xi+\eta)a\pi_n + \pi_m |D(Q(a))|\pi_n + \pi_m a(\xi+\eta).
		\end{split}
		\]
		If $m,n,r,s\in\Bbb N$ are such that $r\geq m$ and $s\geq n$, then $u_m\pi_{k_r} = u_m$ and $\pi_{k_s}v_n = v_n$, so it follows from the last inequality that
		\[
		\begin{split}
			|u_mD(\pi_{k_r}t\pi_{k_s})v_n| &= |u_m\pi_{k_r}D(\pi_{k_r}t\pi_{k_s})\pi_{k_s}v_n|\leq u_m|\pi_{k_r}D(\pi_{k_r}t\pi_{k_s})\pi_{k_s}|v_n \\
			&\leq u_m(\xi+\eta)av_n + u_m|D(Q(a))|v_n + u_ma(\xi+\eta)v_n.
		\end{split}
		\]
		This establishes (\ref{firststep}) and completes the proof of the theorem.
	\end{proof}
	
	We are now in a position to establish the main result of this section -- Theorem~\ref{atomicmain} below. The theorem will give conditions on a Banach lattice algebra of regular operators $\cal A$ and on a Banach lattice module $\frak Y$ over that algebra, under which every bounded derivation from $\cal A$ to $\frak Y$ is regular. Before we can state the theorem, though, some more terminology needs to be introduced.  
	
	Given a Banach algebra $\cal A$ and a Banach $\cal A$-module $\frak X$, it is well-known that its Banach dual $\frak X'$, endowed with module products, defined by  
	\[
	(ax')(x) := x'(xa) \quad\text{ and }\quad (x'a)(x) := x'(ax) \qquad (a\in\cal A,\,x\in\frak X,\,x'\in\frak X'), 
	\]
	becomes a Banach $\cal A$-module. (One such module is usually called a dual Banach module.) If $\cal A$ is a Banach lattice algebra and $\frak X$ is a Banach lattice $\cal A$-module, then $\frak X'$ with its usual order and the above module products, is readily seen to be a Banach lattice $\cal A$-module, in the sense of the note. In what follows, given a Banach lattice algebra $\cal A$ and a Banach lattice $\cal A$-module $\frak X$, the Banach dual $\frak X'$ will be always thought of as a Banach lattice $\cal A$-module in this way. 
	
	As customary, we shall say that sequences $(x_i)$ and $(y_i)$ (resp.~$(x_i)_{i=1}^n$ and $(y_i)_{i=1}^n$) of a Banach lattice $X$ are equivalent if there is a constant $L>0$ so that
	\[
	L^{-1}\Big\|\sum\nolimits_i \alpha_i x_i\Big\|\leq \Big\|\sum\nolimits_i \alpha_i y_i\Big\|\leq L\Big\|\sum\nolimits_i \alpha_i x_i\Big\|,
	\]
	for every $(\alpha_i)\in c_{00}$ (resp.~every $n$-tuple of scalars $(\alpha_i)_{i=1}^n$). We shall indicate this by writing $(x_i)\overset{L}{\sim} (y_i)$ or simply $(x_i)\sim (y_i)$ (resp.~$(x_i)_{i=1}^n\overset{L}{\sim} (y_i)_{i=1}^n$). 
		
	Following \cite{Bl1}, we shall say a separable purely atomic Banach lattice $X$ satisfies $(\star)$ if \vspace{1mm}
	
	\begin{itemize}
		\item[] \emph{for some $\mu\geq 0$ there is a sequence $(\mu_n)\subset [1,\mu]$ and a sequential arrangement of its normalized atoms, $(x_i)$ say, such that, for every $n\in\Bbb N$, if $l_1<l_2<\cdots<l_n$ are such that $(x_i)_{i=1}^n\overset{\mu_n}{\sim} (x_{l_i})_{i=1}^n$ then for every $k\in\Bbb N$ there exists $l_{n+1}\geq l_n+k$ such that $(x_i)_{i=1}^{n+1}\overset{\mu_{n+1}}{\sim} (x_{l_i})_{i=1}^{n+1}$.}
	\end{itemize}
	
	\vspace{1mm}
	The class of Banach lattices that satisfy $(\star)$ contains not only Banach lattices whose normalized atoms form subsymmetric bases, but also any direct sum $(\bigoplus_i X_i)_{\mathbf{e}}$, with $\mathbf{e} = (e_i)_{i=1}^{\frak m}$ ($\frak m\in\Bbb N\cup \{\infty\}$) a 1-unconditional basis for some Banach space $E$, and $(X_i)_{i=1}^{\frak m}$ either a sequence of Banach lattices satisfying $(\star)$ with $\mu = 1$, or a sequence of Banach lattices whose normalized atoms form subsymmetric bases with subsymmetric-constants no greater than some constant $\mu$ (here, as customary, $(\bigoplus_i X_i)_{\mathbf{e}}$ stands for the space $\{(x_i)\in\prod_i X_i : \sum_i \|x_i\|e_i\in E\}$, endowed with the norm $\|(x_i)\| := \|\sum_i \|x_i\|e_i\|_E$ $((x_i)\in(\bigoplus_i X_i)_{\mathbf{e}})$ and the order induced by the cone $(\bigoplus_i X_i)_{\mathbf{e}}\bigcap (\prod_i X_i^+)$). 
	
	Given a sum $(\bigoplus_n X_n)_\mathbf{e}$, in which every $X_n$ satisfies $(\star)$ for the same constant $\mu$, one can define a Schauder basis $(x_i)$ for it as follows. Let $I$ be the index set of the elements in the basis~$\mathbf{e}$. First, for each $X_n$ $(n\in I)$ choose a basis $(x_{n,i})$ satisfying $(\star)$ with constant $\mu$; next, choose a bijective map $\phi:I\times\Bbb N\to\Bbb N$, order preserving on every subset of the form $\{n\}\times\Bbb N$ $(n\in I)$; and then define $(x_i)$ by $x_i := x_{\phi^{-1}(i)}$ $(i\in\Bbb N)$. To simplify the statements of our results, we shall refer to one such basis as a {\bf basis for $(\bigoplus_n X_n)_\mathbf{e}$, satisfying $(\star)$ with constant $\mu$ in each summand}. 
	
	Given $\xi\in X$ and $\lambda\in X'$, we write $\lambda\otimes\xi$ for the rank-one operator $x\mapsto \lambda(x)\xi$, $X\to X$. Lastly, recall that if $X$ and $Y$ are Riesz spaces, $Y$ is order complete and $\cal F\subseteq\cal L^r(X,Y)$ is an upward directed family such that $\{T(x) : T\in\cal F\}$ is bounded above for every $x\in X_+$, then $\sup{\cal F}$ exists and satisfies $(\sup{\cal F})(x) = \sup\{T(x) : T\in\cal F\}$ $(x\in X_+)$ (see for instance \cite[Theorem~1.19]{AB}).
	
	\begin{thm}\label{atomicmain}
		Let $(X_i)_{i=1}^{\frak m}$ be a sequence of separable atomic Banach lattices satisfying $(\star)$ for the same constant $\mu$, let $\mathbf{e}$ be a 1-unconditional basis with $|\mathbf{e}| = \frak m$ and let $X := (\bigoplus_i X_i)_{\mathbf{e}}$. Let $\cal A$ be a closed Riesz algebra ideal of $\cal L^r(X)$ such that $\cal A\cal A = \cal A$, and let $\frak X$ be a Banach lattice $\cal A$-module. Let $(x_i)$ be a basis for $X$ satisfying $(\star)$ with constant $\mu$ in each summand, let $(x_i^*)$ be the corresponding sequence of biorthogonal functionals, let $q_i := x_i^*\otimes x_i$ $(i\in\Bbb N)$ and let $e_n := \sum_{i=1}^n q_i$ $(n\in\Bbb N)$. Suppose 
		\vspace{1mm}
		\begin{itemize}
			\item[--] $(x_i)$ is shrinking; 
			
			\item[--] the linear maps defined by $Ex' := \sup_n e_nx'$ and $Fx' := \sup_n x'e_n$ $(x'\in\frak X'_+)$ have weak*-closed ranges;
			
			\item[--] $\sup_n e_nax' = ax'$ and $\sup_n x'ae_n = x'a$ $(a\in\cal A_+,\, x'\in\frak X'_+)$; 
			
			\item[--] at least one of the following holds:
			\begin{itemize}
				\item[(i)] for each $i\in\Bbb N$, if $(\xi_n)\subset\frak X'$ is such that either $\xi_n\in q_{n}\frak X'q_{i}$ $(n\in\Bbb N)$ or $\xi_n\in q_{i}\frak X'q_{n}$ $(n\in\Bbb N)$, then $(\xi_n)$ has no subsequence equivalent to the unit vector basis of $c_0$; or
								
				\item[(ii)] $X$ is reflexive and $\dim{q_{i}\frak X'q_{j}}<\infty$ for some (hence, for every) pair $i,j\in\Bbb N$; 
				
				\item[(iii)] every bounded derivation from $\cal A^r(X)$ to $\frak X'$ is inner, i.e., of the form $a\mapsto ax'-x'a$, $\cal A^r(X)\to\frak X'$, for some $x'\in\frak X'$ fixed. \vspace{1mm}
			\end{itemize}
		\end{itemize}
		Then every bounded derivation $D:\cal A\to\frak X'$ is regular.
	\end{thm}
	
	\begin{rem}
		That $\dim{q_i\frak X'q_j}$ is independent of $i,j\in\Bbb N$ follows on noting that each map $\xi_{ij}'\mapsto (x_i^*\otimes x_k)\xi_{ij}'(x_l^*\otimes x_j)$, $q_i\frak X'q_j\to q_k\frak X'q_l$ $(i,j,k,l\in\Bbb N)$ is an isometric isomorphism.
	\end{rem}
	
	\begin{rem}
		We should point out that the conditions under the second and third items of Theorem~\ref{atomicmain}, are in fact independent. For instance, let $X$ be a Banach lattice with a Schauder basis $(x_i)$ formed by its normalized atoms. If $\cal A = \cal L^r(X)$ and $\frak X = \cal L^r(X)/\cal A^r(X)$, then $e_nx' = 0 = x'e_n$ $(x'\in\frak X',\,n\in\Bbb N)$, so $E(\frak X') = F(\frak X') = \{0\}$, while $\cal A\frak X' = \frak X' = \frak X'\cal A$, i.e., the condition of the second item holds while that of the third item do not. In the opposite direction, let $X$ be as before but suppose in addition that $(x_i)$ is shrinking. Let $\cal A = \cal A^r(X)$ and $\frak X = \cal L^r(X)$. The sequence $(e_n)$ of natural projections with respect to $(x_i)$ is then a norm-bounded a.i.~for $\cal A^r(X)$. As a result, for every $a\in\cal A$ and $x'\in\frak X'$, $\lim_n e_nax' = ax'$ and $\lim_n x'ae_n = x'a$ -- both limits in the norm-topology -- so we must have $\sup_n e_nax' = ax'$ and $\sup_n x'ae_n = x'a$, whenever $a$ and $x'$ are positive. On the other hand, $E(\frak X')$ and $F(\frak X')$ are not weak*-closed in $\frak X'$. To see it consider, for instance, the sequence $(\lambda_n)$ in $E(\frak X')\cap F(\frak X')$, defined by $\lambda_n(a) := n^{-1} \sum_{i=1}^n x_i^*(a(x_i))$ $(a\in\frak X,\,n\in\Bbb N)$. Let $\lambda$ be a weak*-limit point of $\{\lambda_n : n\in\Bbb N\}$ and let $(\lambda_{n_\alpha})$ be a subnet of $(\lambda_n)$, with weak*-limit $\lambda$. Then $\lambda(\mathrm{id}_X) = \lim_\alpha \lambda_{n_\alpha}(\mathrm{id}_X) = 1$, while on the other hand $E\lambda = w^*\text{-}\lim_n e_n\lambda$ (see the proof below), so $(E\lambda)(\mathrm{id}_X) = \lim_n (e_n\lambda)(\mathrm{id}_X) = \lim_n \lambda(e_n) = \lim_n \lim_\alpha \lambda_{m_\alpha}(e_n) = \lim_n \lim_\alpha m_\alpha^{-1} \sum_{i=1}^n x_i^*(x_i) = 0$. The map $E$ is a projection (see below), and so, $\lambda\notin E(\frak X')$. The proof that $\lambda\notin F(\frak X')$, is completely analogous. 
	\end{rem}
	
	Derivations into dual Banach modules occupy an importance place within the bounded cohomology theory of Banach algebras, so our choice of dual modules in Theorem~\ref{atomicmain} is partially justified by this. In addition, in the case of a dual module, the weak*-topology provides a perfect candidate for the topology $\widetilde{\omega}$ of Theorem~\ref{32like}. It is conceivable that one may be able to replace dual Banach lattice modules in Theorem~\ref{atomicmain} by more general Banach lattice modules, satisfying the Levi or some similar property, but we will not explore that possibility here. 
	
	\begin{proof}[Proof of Theorem~\ref{atomicmain}]
		Let $\cal A$ and $\frak X$ be as in the hypotheses and let $D:\cal A\to\frak X'$ be a continuous derivation. We shall define topologies $\tau$ and $\omega$ on $\cal A$ and $\frak X'$, respectively, and show that $(\cal A,\tau)$, $(\frak X',\omega)$ and $D:(\cal A,\tau)\to (\frak X,\omega)$ satisfy the hypotheses of Theorem~\ref{32like}. We shall suppose $\|ax'\|\leq M\|a\|\|x'\|$ and $\|x'a\|\leq M\|a\|\|x'\|$ $(a\in\cal A,\, x'\in\frak X')$. 
		
		\smallskip
		
		$\bullet$ \emph{The topology $\tau$ and the algebra $(\cal A,\tau)$:} We let $\tau$ be the strict topology on $\cal A$ define by $\cal F(X)$, i.e., for each $b\in\cal F(X)$, let $\tau_b:\cal A\to\Bbb R_+\cup\{0\}$ be the seminorm defined by
		\[
		\tau_{b}(a) := \|ab\| + \|ba\|
		\quad (a\in\cal A),
		\]
		and let $\tau$ be the topology on $\cal A$ generated by the $\tau_b$'s.	It is easy to see that multiplication on $(\cal A,\tau)$ is separately continuous. Furthermore, $(\cal A,\tau)$ is Hausdorff, for if $a\in\cal A\setminus\{0\}$, then $|a|(x_i)\ne 0$ for some $i$, and hence, $\tau_{q_i}(a)>0$. So $(\cal A,\tau)$ is definitely a topological algebra in the sense of the note. That $\cal A$ is semiprime follows from \cite[Theorem~2.5.8(i)]{Da} and \cite[Proposition 30.5]{BD}.
		
		Next, let $(a_\alpha)\subset\cal A_+$ be a $\tau$-converging net with $\tau$-limit $a$. Since $\|a_\alpha(x_i)-a(x_i)\| = \|a_\alpha q_i-aq_i\|\leq \tau_{q_i}(a_\alpha-a)\to 0$ as $n\to\infty$ $(i\in\Bbb N)$, one has that $a(x_i)\geq 0$ $(i\in\Bbb N)$, so $a\in\cal A_+$ and $\cal A_+$ is therefore $\tau$-closed. Furthermore, $(e_n)\subset\cal A_+$ is an a.i.~for $(\cal A,\tau)$. To see it, let $a\in\cal A$ and $b\in\cal F(X)$ arbitrary. Then, for every $n\in\Bbb N$,
		\[
		\begin{split}
			\tau_{b}(a-ae_n) &= \|(a-ae_n)b\| + \|b(a-ae_n)\| \\ 
			&= \|a\|\|b-e_nb\| + \|ba-(ba)e_n\|.
		\end{split}
		\]
		Since $(x_i)$ is shrinking, $(e_n)$ is an a.i.~for $\cal F(X)$, so $\tau_{b}(ae_n-a)\to 0$ as $n\to\infty$. The proof that $\tau_{b}(e_na-a)\to 0$ as $n\to\infty$ is completely analogous. 
		
		\smallskip
		
		$\bullet$ \emph{The topology $\omega$ and the $(\cal A,\tau)$-module $(\frak X',\omega)$:}
		Before we can define $\omega$, some preparation is needed. Let $E$ and $F$ be as in the hypotheses. We show first that $E$, $F$ is a pair of (necessarily, bounded) commuting projections. To this end, recall that $\cal L^r(\frak X')$ can be isometrically identified with $(\frak X'\otimes_{|\pi|}\frak X)'$, and so, norm-bounded subsets of $\cal L^r(\frak X')$ are pre-compact with respect to the topology generated by the seminorms $\|\cdot\|_{x,x'}:\cal L^r(\frak X')\to\Bbb R_+$, $T\mapsto|(Tx')(x)|$ $(x\in\frak X,\,x'\in\frak X')$. It is then easy to see that the (necessarily unique) limit-points in $\cal L^r(\frak X')$ of the sets $\{L_n:\frak X'\to\frak X',x'\mapsto e_nx' : n\in\Bbb N\}$ and $\{R_n:\frak X'\to\frak X',x'\mapsto x'e_n : n\in\Bbb N\}$, with respect to this topology, must coincide with $E$ and $F$, respectively. Accordingly, we must have that 
		\begin{equation}\label{weakstardef}
			(Ex')(x) = \lim_n (e_nx')(x)
			\quad\text{ and }\quad
			(Fx')(x) = \lim_n (x'e_n)(x)
			\quad (x\in\frak X,\, x'\in\frak X').
		\end{equation}
		It follows that
		\[
		\begin{split}
			E(Ex') &= w^*\text{-}\lim_n e_n(w^*\text{-}\lim_m e_mx') \\
			&= w^*\text{-}\lim_n (w^*\text{-}\lim_m e_ne_mx') = w^*\text{-}\lim_n e_nx' = Ex' \quad (x'\in\frak X'),
		\end{split}
		\]
		and similarly, that $F^2 = F$. Also from (\ref{weakstardef}), taking into account that $x'\mapsto ax'$, $\frak X'\to\frak X'$ and $x'\mapsto x'a$, $\frak X'\to\frak X'$ are $w^*$-$w^*$-continuous $(a\in\cal A)$ and that $\frak X'_+$ is $w^*$-closed, one finds that if $x'\in\frak X'_+$ then
		\[
		E(Fx') = w^*\text{-}\lim_m w^*\text{-}\lim_n e_mx'e_n\leq w^*\text{-}\lim_m w^*\text{-}\lim_n (Ex')e_n = F(Ex').
		\]
		Similarly, $F(Ex')\leq E(Fx')$ $(x'\in\frak X'_+)$, and so, $EF = FE$. Note that $E(\frak X')+F(\frak X')$ is then the sum of the ranges of the mutually orthogonal projections $E-EF$, $EF$ and $F-EF$, and is therefore norm-closed.
		
		We now define $\omega$ as follows. For every $\mathbf{b} = (b_1,b_2)\in\cal F(X)\times\cal F(X)$ and every $\mathbf{f} = (f_1,f_2,f_3,f_4)\in\frak X\times (\frak X\cap F(\frak X')^\perp)\times (\frak X\cap E(\frak X')^\perp)\times (E(\frak X')+F(\frak X'))^\perp$ (where we have identified $\frak X$ with its image in $\frak X''$ under the canonical embedding), let
		\[
		\omega_{\mathbf{b},\mathbf{f}}(x')
		:= |f_1(b_1x'b_2)| + |f_2(b_1x')| + |f_3(x'b_2)| + |f_4(x')|
		\quad (x'\in\frak X'),
		\]
		and let $\omega$ be the topology on $\frak X'$ generated by the $\omega_{\mathbf{b},\mathbf{f}}$'s. For each seminorm $\omega_{\mathbf{b},\mathbf{f}}$, 
		\[
		\begin{split}
			\omega_{\mathbf{b},\mathbf{f}}(ax')
			&= |f_1(b_1ax'b_2)| + |f_2(b_1ax')|
			\leq M\|x'\|(\|f_1\|\|b_2\|+\|f_2\|)\tau_{b_1}(a),
		\end{split}
		\]
		and similarly,
		\[
		\omega_{\mathbf{b},\mathbf{f}}(x'a)
		\leq M\|x'\|(\|f_1\|\|b_1\|+\|f_3\|)\tau_{b_2}(a).
		\]
		So if $(a_\alpha)\subset\cal A$ is a net such that $\tau$-$\lim_\alpha a_\alpha = 0$ then $\omega$-$\lim_\alpha a_\alpha x' = \omega$-$\lim_\alpha x'a_\alpha = 0$ $(x'\in\frak X')$. Also,
		\[
		\omega_{\mathbf{b},\mathbf{f}}(ax')\leq
		\omega_{(b_1a,b_2),\mathbf{f}}(x')
		\quad\text{and}\quad
		\omega_{\mathbf{b},\mathbf{f}}(x'a)\leq
		\omega_{(b_1,ab_2),\mathbf{f}}(x')
		\quad (x'\in\frak X'),
		\]
		so, if $(x'_\alpha)\subset\frak X'$ is a net such that $\omega$-$\lim_\alpha x'_\alpha = 0$, then $\omega$-$\lim_\alpha ax'_\alpha = \omega$-$\lim_\alpha x'_\alpha a = 0$ for every~$a\in\cal A$. 
		
		To conclude $(\frak X',\omega)$ is a topological $(\cal A,\tau)$-module, it only remains to check that $(\frak X',\omega)$ is Hausdorff. To see it, we consider several cases. If $x'\in\frak X'\setminus (E(\frak X')+F(\frak X'))$, then there is $f\in (E(\frak X')+F(\frak X'))^\perp$ such that $f(x')\ne 0$ (recall $E(\frak X') + F(\frak X')$ is norm-closed). So suppose $x'\in (E(\frak X')+F(\frak X'))\setminus \{0\}$. Then $x' = (E-EF)x' + EFx' + (F-EF)x'$. If $(E-EF)x'\ne 0$ there is $x\in\frak X\cap F(\frak X')^\perp$ such that $x'(x) = ((E-EF)x')(x)\ne 0$ and $x'(x) = (Ex')(x) = \lim_n (e_nx')(x)$, so $(e_nx')(x)\ne 0$ for some $n\in\Bbb N$. Analogously, if $(F-EF)x'\ne 0$ there are $x\in\frak X\cap E(\frak X')^\perp$ and $m\in\Bbb N$ such that $(x'e_m)(x)\ne 0$. Lastly, if $x' = EFx' = w^*\text{-}\lim_n w^*\text{-}\lim_n e_nx'e_m$, there are $x\in\frak X$ and $m,n\in\Bbb N$ such that $(e_nx'e_m)(x)\ne 0$. Thus, for any $x'\in\frak X'\setminus \{0\}$, there is a seminorm $\omega_{\mathbf{b},\mathbf{f}}$ such that $\omega_{\mathbf{b},\mathbf{f}}(x')\ne 0$.
		
		\smallskip
		
		$\bullet$ \emph{Sequential continuity of $D:(\cal A,\tau)\to (\frak X',\omega)$:} We will show here that $D$ is in fact $\tau$-$\omega$-continuous. Note $D(\cal A)\subseteq \cal A\frak X'+\frak X'\cal A\subseteq E(\frak X')+F(\frak X')$ (the first inclusion because $\cal A = \cal A\cal A$, and the second because $ax' = \sup_n e_nax' = E(ax')$ and $x'a = \sup_n x'ae_n = F(x'a)$ $(a\in\cal A_+,\, x'\in\frak X'_+)$ -- by hypothesis). Thus, for every seminorm $\omega_{\mathbf{b},\mathbf{f}}$, one has that
		\[
		\begin{split}
			\omega_{\mathbf{b},\mathbf{f}}(D(a))
			&= |f_1(b_1D(a)b_2)|
			+ |f_2(b_1D(a))| + |f_3(D(a)b_2)| \\
			&= \big|f_1(D(b_1ab_2)) -
			f_1(D(b_1)ab_2) - f_1(b_1aD(b_2))\big| \\
			&\hspace{5cm} + |f_2(D(b_1a))| + |f_3(D(ab_2))| \\
			&\leq 2M\|f_1\|\|D\|\|b_2\|\tau_{b_1}(a) +
			M\|f_1\|\|D\|\|b_1\|\tau_{b_2}(a) \\
			&\hspace{3.9cm} + \|f_2\|\|D\|\tau_{b_1}(a) +
			\|f_3\|\|D\|\tau_{b_2}(a).
		\end{split}
		\]
		It follows readily from this that $D:(\cal A,\tau)\to(\frak X',\omega)$ is continuous at 0, and hence, continuous.
		
		\smallskip
		
		$\bullet$ \emph{The topology $\widetilde{\omega}$:} We let $\widetilde{\omega}$ be the weak*-topology on $\frak X'$. Then $\widetilde{\omega}$ meets all the requirements of Theorem~\ref{32like}, i.e., $\frak X'_+$ is weak*-closed; absolutely majorized subsets of $\frak X'$ (being norm-bounded) are weak*-precompact; for every $a\in\cal A$, the maps $x'\mapsto ax'$, $\frak X'\to\frak X'$ and $x'\mapsto ax'$, $\frak X'\to\frak X'$ are $w^*$-$w^*$-continuous; and last, if $x'\in E(\frak X')+F(\frak X')$, then, for every seminorm $\omega_{\mathbf{b},\mathbf{f}}$, one has that $\omega_{\mathbf{b},\mathbf{f}}(x') = |x'(b_2f_1b_1)| + |x'(f_2b_1)| + |x'(b_2f_3)|$, so the weak*-topology is certainly finer than $\omega$ on $E(\frak X')$ and $F(\frak X')$. (Note that $E(\frak X')$ and $F(\frak X')$ are the weak*-closures of $\cal A\frak X'$ and $\frak X'\cal A$, respectively, for $\cal A\frak X'\subseteq E(\frak X')\subseteq \overline{\cal A\frak X'}^{w^*}$ and $\frak X'\cal A\subseteq F(\frak X')\subseteq \overline{\frak X'\cal A}^{w^*}$, where the second inclusion follows, in both cases, from (\ref{weakstardef}).) 
		
		\smallskip
		
		$\bullet$ \emph{The sequence $(p_i)$:} Lastly, we construct a sequence $(p_i)$ of pairwise orthogonal o-minimal idempotents, so that the corresponding sequence $\pi_n := \sum_{i=1}^n p_i$ $(n\in\Bbb N)$, satisfies the requirements of Theorem~\ref{32like}.
		
		\smallskip
		
		-- \emph{Suppose first that (i) holds:} 
		Set $\xi_{ij}' := q_i|q_iD(q_i)q_j|q_j$ $(i,j\in\Bbb N)$. For each $i\in\Bbb N$, let $J_i := \{j : \xi_{ij}'\ne 0\}$ and note that for every $(\alpha_j)\in c_{00}$, $(\varepsilon_j)\in\{\pm 1\}^{\Bbb N}$ and $m,n\in\Bbb N$, with $m\leq n$,
		\[
		\bigg\|\sum_{j\in J_i, j\leq m} \alpha_j\varepsilon_j\xi_{ij}'\bigg\|
		= \bigg\|\bigg(\sum_{j\in J_i, j\leq n} \alpha_j\xi_{ij}'\bigg)
		\bigg(\sum_{j=1}^m \varepsilon_jq_j\bigg)\bigg\|
		\leq M\bigg\|\sum_{j\in J_i, j\leq n} \alpha_j\xi_{ij}'\bigg\|\|e_m\|,
		\]
		so $(\xi_{ij}')_{j\in J_i}$ is an unconditional basic sequence in $\frak X'$ with $\xi'_{ij}\in q_i\frak X'q_j$ $(j\in J_i)$. Since no such sequence can be equivalent to the unit vector basis of $c_0$ (by our hypotheses) and
		\[
		\begin{split}
			\sup_{|t_j|\leq 1 : 1\leq j\leq n}
			\Big\|\sum\nolimits_{j\in J_i, j\leq n} t_j\xi_{ij}'\Big\| &=
			\Big\|\sum\nolimits_{j\in J_i, j\leq n} \xi_{ij}'\Big\| \\
			&\leq \|q_i|D(q_i)|e_n\|\leq M\|q_i\|\|D(q_i)\|\|e_n\|
			\quad (n\in\Bbb N),
		\end{split}
		\]
		we must have $\lim_{j} \|\xi_{ij}'\| = 0$ $(i\in\Bbb N)$ (see for instance, \cite[Chapter~V, Theorem~6 and Corollary~7]{Di}). In a similar fashion, but working instead with the columns of $(\xi_{ij}')_{i,j\in\Bbb N}$, noting that $q_iD(q_i)q_j = - q_iD(q_j)q_j$, and hence, that $\xi_{ij}' = q_i|q_iD(q_j)q_j|q_j$ $(i,j\in\Bbb N)$, one shows that $\lim_{i} \|\xi_{ij}'\| = 0$ $(j\in\Bbb N)$. 
		
		One can now construct, inductively, a strictly increasing sequence, $(k_i)$ say, so that the double sum $\sum_i \sum_j \|\xi_{k_ik_j}'\|$ converges and $(x_{k_i})\sim (x_i)$. This can be done, for instance, as follows. First, to fix ideas, let $\phi:I\times\Bbb N\to\Bbb N$ and $(x_{n,j})$ $(n\in I = \Bbb N)$ be the bijective map and Schauder bases, used in the definition of $(x_i)$; and for each $n\in I$, let $(\mu_{n,j})\subset [1,\mu]$ be a sequence, as in the definition of the $(\star)$ property, for the basis $(x_{n,j})$. Set $k_1 := 1$, and in general, if $k_1,\ldots,k_i$ have been chosen and $\phi^{-1}(i+1) = (n,m)$, choose $k_{i+1}\in\phi(\{n\}\times\Bbb N)$ so that $k_{i+1}>k_i$, $\sum_{l=1}^{i} (\|\xi'_{k_{i+1}k_{l}}\| + \|\xi'_{k_{l}k_{i+1}}\|)\leq 2^{-i}$ and $(x_{\phi(n,l)})_{l=1}^m\overset{\mu_{n,m}}{\sim} (x_{k_{\phi(n,l)}})_{l=1}^m$ (which is possible by the result from the previous paragraph and our choice of $(x_{n,j})$). Clearly, the basic sequence $(x_{k_i})$, constructed in this way, will satisfy $(x_{k_i})\overset{\mu}{\sim} (x_i)$. Set $p_i := q_{k_{i}}$ $(i\in\Bbb N)$ and $\pi_n := \sum_{i=1}^n p_i$ $(n\in\Bbb N)$. Then $(p_i)$ is a sequence of pairwise orthogonal o-minimal idempotents, and
		\[
		\sum_{i=1}^n |p_iD(p_i)\pi_n|\leq \sum_{i=1}^n \sum_{j=1}^n p_i|p_iD(p_i)p_j|p_j,
		\]
		so $(\sum_{i=1}^n |p_iD(p_i)\pi_n|)_{n\in\Bbb N}$ is norm-bounded, and as $\frak X'$ is Levi, $\sup_n \sum_{i=1}^n |p_iD(p_i)\pi_n|$ exists. Analogously, one verifies that $\sup_n \sum_{i=1}^n |\pi_nD(p_i)p_i|$ exists.
		
		To see $(\pi_n)$ is convergence preserving, let $(a_n)\subset\cal A$ be a sequence in $\cal A$ with $\tau$-limit $a$ and let $p := \sup_n \pi_n$ in $\cal L^r(X)$. It will suffice to show that $\lim_m \tau_b(\pi_ma_m-pa) = \lim_m \tau_b(a_m\pi_m-ap) = 0$ $(b\in\cal F(X))$. So let $b\in\cal F(X)$ be arbitrary. For every $m,n\in\Bbb N$, one has that
		\begin{equation}\label{convpres}
			\begin{split}
				\tau_{b}(\pi_ma_m-pa)&= \|(\pi_ma_m-pa)b\| + \|b(\pi_ma_m-pa)\| \\
				&\leq \|(\pi_ma_m-pa)(b-e_nb)\| + \|(\pi_ma_m-pa)e_n\|\|b\| \\
				&\hspace{1cm} + \|(b-be_n)(\pi_ma_m-pa)\| + \|b\|\|e_n(\pi_ma_m-pa)\|.
			\end{split}
		\end{equation}
		Note that
		\[
		\|(\pi_ma_m-pa)e_n\|\leq \sum_{i=1}^n (\|\pi_m\|\|(a_m-a)q_i\| + \|(\pi_m-p)aq_i\|),
		\]
		and
		\[
		\|e_n(\pi_ma_m-pa)\|\leq \sum_{i=1}^n \|q_i(a_m-a)\|,
		\]
		so both, $\|(\pi_ma_m-pa)e_n\|$ and $\|e_n(\pi_ma_m-pa)\|$ tend to 0 as $m\to\infty$ for each $n\in\Bbb N$ fixed (that $\lim_m \|(\pi_m-p)aq_i\| = 0$ $(i\in\Bbb N)$ follows on noting that $\|(\pi_m-p)aq_i\| = \|(\pi_m-p)ax_i\|$ and $\|(p-\pi_m)x\|\to 0$ as $m\to\infty$ for every $x\in X$, because $(x_i)$ is an unconditional basis for $X$). Now let $\varepsilon>0$ arbitrary and choose $n$ large enough so that $\max\{\|b-e_nb\|,\|b-be_n\|\}\leq \varepsilon/(2(\|a\|+\sup_m \|a_m\|))$. Then, letting $m\to\infty$ on the right hand side of (\ref{convpres}), one obtains that $\limsup_m \tau_b(\pi_ma_m-pa)\leq \varepsilon$. That $\tau$-$\lim_m \pi_ma_m = pa$ follows readily from this. The proof that $\tau$-$\lim_m a_m\pi_m = ap$ is completely analogous. 
		
		Next we check $(e_n)$ admits a factorization through $(\pi_n)$. To this end, let $u_n := \sum_{i=1}^n x_{k_i}^*\otimes x_i$ and $v_n := \sum_{i=1}^n x_i^*\otimes x_{k_i}$ $(n\in\Bbb N)$. Since $(x_i)\overset{\mu}{\sim} (x_{k_i})$, $(u_n)$ and $(v_n)$ are norm-bounded (and hence, $\tau$-bounded) sequences in $\cal F(X)$. Furthermore, $e_n = u_n\pi_nv_n$ $(n\in\Bbb N)$, and clearly, $(u_n),(v_n)\subset\cal A_+$. To see they are convergence preserving, let $(a_n)\subset\cal A$ be a $\tau$-convergent sequence with $\tau$-limit $a$. Let $u,v\in\cal L^r(X)$ be defined by $u(\sum_i \alpha_ix_i) := \sum_i \alpha_{k_{i}}x_i$ and $v(\sum_i \alpha_ix_i) := \sum_i \alpha_ix_{k_{i}}$ $(\sum_i \alpha_ix_i\in X)$. Then, for every $b\in\cal F(X)$, with $p$ as in the previous paragraph, one has that
		\[
		\begin{split}
			\tau_{b}(u_na_n-ua) &= \tau_{b}(u(\pi_na_n-pa))\leq \|u\|\tau_{b}(\pi_na_n-pa) + \tau_{bu}(\pi_na_n-pa), \\
			\tau_{b}(a_nu_n-au) &= \tau_{b}((a_ne_n-a)u)\leq \tau_{ub}(a_ne_n-a) + \|u\|\tau_{b}(a_ne_n-a), \\
			\tau_{b}(v_na_n-va) &= \tau_{b}(v(e_na_n-a))\leq \|v\|\tau_{b}(e_na_n-a) + \tau_{bv}(e_na_n-a), \\
			\tau_{b}(a_nv_n-av) &= \tau_{b}((a_n\pi_n-ap)v)\leq \tau_{vb}(a_n\pi_n-ap) + \|v\|\tau_{b}(a_n\pi_n-ap).
		\end{split}
		\]
		Note that $\lim_n \tau_c(a_ne_n-a) = \lim_n \tau_c(e_na_n-a) = 0$ $(c\in\cal F(X))$, for
		\[
		\begin{split}
			\tau_c(a_ne_n-a)&\leq \|c(a_n-a)e_n\| + \|(a_n-a)e_nc\| + \tau_c(ae_n-a) \\
			&\leq \tau_c(a_n-a) + \|(a_n-a)(e_nc-c)\| + \tau_c(ae_n-a),
		\end{split}
		\]
		and similarly,
		\[
		\tau_c(e_na_n-a)\leq \tau_c(a_n-a) + \|(ce_n-c)(a_n-a)\| + \tau_c(e_na-a).
		\]
		Combining this last observation with the previous inequalities and the result from the previous paragraph, one obtains that $u_na_n\to ua$, $a_nu_n\to au$, $v_na_n\to va$ and $a_nv_n\to av$ in $(\cal A,\tau)$, and hence, that $(u_n)$ and $(v_n)$ are both convergence preserving.
		
		It only remains to notice that each one of the sets $\{D(u_n) : n\in\Bbb N\}$, $\{D(v_n) : n\in\Bbb N\}$, $\{u_nx' : n\in\Bbb N\}$ and $\{x'v_n : n\in\Bbb N\}$ $(x'\in\frak X')$ is norm-bounded, and hence, weak*-precompact.
				
		\smallskip
		
		-- \emph{Now suppose it is condition (ii) the one that holds:} It will suffice to show that $\lim_{i} \|\xi_{ij}'\| = \lim_{j} \|\xi_{ij}'\| = 0$ $(i,j\in\Bbb N)$, where the $\xi_{ij}'$'s are defined as before. We argue by contradiction. Suppose first for some $i\in\Bbb N$ there is an infinite subset $J\subseteq\Bbb N$ such that $\inf\{\|\xi_{ij}'\| : j\in J\} =: \delta >0$. Since no subsequence of $(x_{i})$ can be equivalent to the unit vector basis of $\ell_1$, there must be a scalar sequence $(\alpha_i)_{i\in J}$ such that $\sum_{i\in J} |\alpha_i| = \infty$ while $\sum_{i\in J} \alpha_i x_i$ converges. Let $q_{ij} := x_j^*\otimes x_i$ $(i,j\in\Bbb N)$ -- of course, $q_{ii} = q_i$ $(i\in\Bbb N)$. Then, for any $k\in\Bbb N$ fixed, one must have, 
		\begin{equation}\label{(ii)}
			\delta \sum_{j\in J, j\leq n} |\alpha_j|\leq
			\sum_{j\in J, j\leq n} |\alpha_j|\|\xi_{ij}'\|
			\leq \sum_{j\in J, j\leq n} |\alpha_j| M\|\xi_{ij}'q_{jk}\|\|q_{kj}\|
			\quad (n\in\Bbb N).
		\end{equation}
		If $(f_i)_{i=1}^m$ is a basis for $q_i\frak X'q_k$ and $(f_i^*)_{i=1}^m$ is the corresponding sequence of biorthogonal functionals, then, for some constant $C$, 
		\[
		\|\xi_{ij}'q_{jk}\|\leq C \sum_{l=1}^m |f_l^*(\xi_{ij}'q_{jk})|
		\quad (j\in J).
		\]
		Combining the latter with (\ref{(ii)}), one obtains that
		\[
		\delta \sum_{j\in J, j\leq n} |\alpha_j|\leq
		C M\sum_{l=1}^m \sum_{j\in J, j\leq n}
		|\alpha_j| |f_l^*(\xi_{ij}'q_{jk})|,
		\]
		so we must have $\sum_{j\in J} |f_l^*(\alpha_j\xi_{ij}'q_{jk})| = \infty$ for some $1\leq l\leq m$.	For each $j\in J$, set $\theta_j := |f_l^*(\alpha_j \xi_{ij}'q_{jk})|/f_l^*(\alpha_j \xi_{ij}'q_{jk})$ if $f_l^*(\alpha_j \xi_{ij}'q_{jk})\ne 0$ and zero otherwise. Then, for every $n\in\Bbb N$,
		\[
		\begin{split}
			\sum_{j\in J, j\leq n}
			\big|f_l^*(\alpha_j\xi_{ij}'q_{jk})\big|
			&= \Bigg|\sum_{j\in J, j\leq n}
			\theta_j f_l^*(\alpha_j\xi_{ij}'q_{jk})\Bigg|
			\leq \|f_l^*\| \Bigg\|\sum_{j\in J, j\leq n}
			\theta_j\alpha_j\xi_{ij}'q_{jk}\Bigg\| \\
			&= \|f_l^*\| \Bigg\|\bigg(\sum_{j\in J, j\leq n} \xi_{ij}'\bigg)
			\bigg(\sum_{j\in J, j\leq n} \theta_j \alpha_j q_{jk}\bigg)\Bigg\| \\
			&\leq M\|f_l^*\| 
			\|q_i\|\|D(q_i)\|\|e_n\|
			\Bigg\|\sum_{j\in J} \alpha_j x_j\Bigg\|,
		\end{split}
		\]
		which is clearly impossible, since the left hand side tends to infinity with $n$. Thus, $\lim_{j} \|\xi_{ij}'\| = 0$ $(i\in\Bbb N)$.
		
		To show $\lim_j \|\xi_{ji}'\| = 0$ $(i\in\Bbb N)$, assume towards a contradiction that there is $J\subseteq\Bbb N$ infinite such that $\inf\{\|\xi_{ji}'\| : j\in J\} =:\delta>0$. Choose $(\alpha_i)_{i\in J}$ so that $\sum_{i\in J} |\alpha_i| = \infty$ and $\sum_{i\in J} \alpha_ix_i^*$ converges. Fix $k\in\Bbb N$ and choose a basis $(g_i)_{i=1}^m$ for $q_k\frak X'q_i$. Using that $\|\xi_{ji}'\|\leq M\|q_{jk}\|\|q_{kj}\xi_{ji}'\|\leq K M\sum_{l=1}^m |g_l^*(q_{kj}\xi_{ji}')|$ $(j\in J)$, where $K$ is a constant depending only on $(g_i)_{i=1}^m$, one obtains that $\delta \sum_{j\in J,\, j\leq n} |\alpha_j|\leq KM\sum_l \sum_{j\in J,\, j\leq n} |\alpha_j||g_l^*(q_{kj}\xi_{ji}')|$, so $\sum_{j\in J} |g_l^*(\alpha_j q_{kj}\xi_{ji}')| = \infty$ for some $1\leq l\leq m$. From this last, arguing along the same lines as above, one concludes that $\sum_{j\in J,\, j\leq n} |g_l^*(\alpha_j q_{kj}\xi_{ji}')|\leq \|g_l^*\|\|\sum_{i\in J} \alpha_ix_i^*\|\|\sum_{j\in J,\, j\leq n} \xi_{ji}'\|$ $(n\in\Bbb N)$ -- whence a contradiction.

		\smallskip
		
		-- \emph{Lastly, consider the case where (iii) holds:} Suppose $D(a) = ax'-x'a$ $(a\in\cal A^r(X))$, for some $x'\in\frak X'$ fixed. In this case, we can take $p_i := q_i$ $(i\in\Bbb N)$. Then $\pi_n = e_n$ $(n\in\Bbb N)$, so $(\pi_n)$ is convergence preserving, and $u_n := e_n$ and $v_n := e_n$ $(n\in\Bbb N)$, provide a factorization of $(e_n)$ through itself with the required properties. Furthermore, 
		\[
		\sum_{i=1}^n |p_iD(p_i)\pi_n| = \sum_{i=1}^n |p_i(p_ix'-x'p_i)\pi_n|\leq \sum_{i=1}^n p_i|x'|(\pi_n-p_i)\leq \pi_n|x'|\pi_n. 
		\]
		So $\big(\sum_{i=1}^n |p_iD(p_i)\pi_n|\big)_{n\in\Bbb N}$ is norm-bounded, and hence (since every dual Banach lattice has the Levi property), $\sup_n |p_iD(p_i)\pi_n|$ exists. The proof that $\sup_n |\pi_nD(p_i)p_i|$ exists is completely analogous. 
		
		\smallskip
		
		We have thus shown that all hypotheses of Theorem~\ref{32like} are satisfied, and so, the regularity of $D$ follows.
	\end{proof}
	
	To conclude the section, we present some concrete situations in which the conditions of Theorem~\ref{atomicmain} hold.  
	
	\begin{cor}\label{cor1}
		Let $X$, $(x_i)$ and $(q_i)$ be as in Theorem~\ref{atomicmain}. Let $e := \mathrm{id}_X$ and let $\frak X$ be a Banach lattice $\cal L^r(X)$-module such that $\sup_n e_nx' = ex'$ and $\sup_n x'e_n = x'e$ $(x'\in\frak X'_+)$. Furthermore, suppose either $\frak X$ contains no complemented copies of $\ell_1$, or $X$ is reflexive and $\dim{q_i\frak X'q_j}<\infty$ $(i,j\in\Bbb N)$. Then every bounded derivation $D:\cal L^r(X)\to\frak X'$ is regular.
	\end{cor}
	
	\begin{proof}
		The fact that $\sup_n e_nx' = ex'$ and $\sup_n x'e_n = x'e$ for every $x'\in\frak X'_+$, readily implies that $\sup_n e_nax' = ax'$ and $\sup_n x'ae_n = x'a$ $(a\in\cal A_+,\, x'\in\frak X'_+)$. Furthermore, taking into account the definitions of the module products on $\frak X'$, one easily verifies that the ranges of the maps  $x'\mapsto ex'$, $\frak X'\to\frak X'$ and $x'\mapsto x'e$, $\frak X'\to\frak X'$, are both weak*-closed. It only remains to notice that the absence of complemented copies of $\ell_1$ in $\frak X$ would imply condition (i) of Theorem~\ref{atomicmain} (see for instance \cite[Chapter~V, Theorem~10]{Di}). We can thus apply Theorem~\ref{atomicmain}.
	\end{proof}
	
	\begin{cor}\label{cor2} 
		Let $X$, $(x_i)$ and $(q_i)$ be as in Theorem~\ref{atomicmain} and suppose, in addition, that $(x_i)$ is symmetric. Let $\frak X$ be a Banach lattice $\cal L^r(X)$-module such that $\sup_n e_nx' = ex'$ and $\sup_n x'e_n = x'e$ $(x'\in\frak X'_+)$. Then every bounded derivation $D:\cal L^r(X)\to\frak X'$ is regular.	
	\end{cor}
	
	\begin{proof}
		If the basis $(x_i)$ is symmetric and shrinking, then every bounded derivation from $\cal A^r(X)$ into a dual Banach lattice module is inner (the proof of this is basically the same as that of \cite[Proposition~6.1]{Jo} -- the argument given in the latter reference is for $X = \ell_p$ $(1<p<\infty)$ and $\cal A = \cal A(X)$ ($:=$ the norm-closure of $\cal F(X)$ in $\cal B(X)$), but apart from some trivial changes, it works verbatim for any Banach space with a symmetric basis and $\cal A = \cal A^r(X)$), so (iii) holds. The other hypotheses of Theorem~\ref{atomicmain} can be verified as in the proof of Corollary~\ref{cor1}.
	\end{proof}
	
	\begin{rem}
		It is also true that every bounded derivation from $\cal A^r(\ell_1)$ to a dual Banach $\cal A^r(\ell_1)$-module is inner. However, as the unit vector basis of $\ell_1$ is not shrinking, the argument of the proof of Theorem~\ref{atomicmain} does not apply to this case.  
	\end{rem}
	
	\begin{cor}\label{cor3}
		Let $X$, $(x_i)$ and $(q_i)$ be as in Theorem~\ref{atomicmain}. Let $\frak X$ be a Banach lattice $\cal A^r(X)$-module such that $(xe_n)$ and $(e_nx)$ converge for every $x\in\frak X$. Furthermore, suppose either $\frak X$ contains no complemented copies of $\ell_1$, or $X$ is reflexive and $\dim{q_i\frak X'q_j}<\infty$ $(i,j\in\Bbb N)$. Then every bounded derivation $D:\cal A^r(X)\to\frak X'$ is regular.	
	\end{cor}
	
	\begin{proof}
		It follows readily from our hypotheses that the maps $E_0:\frak X\to\frak X$, $x\mapsto\lim_n xe_n$ and $F_0:\frak X\to\frak X$, $x\mapsto\lim_n e_nx$, are well-defined bounded linear maps. Thus, if $E$ and $F$ are as in Theorem~\ref{atomicmain} then $E = E_0'$ and $F = F_0'$, so $E(\frak X')$ and $F(\frak X')$ are both weak*-closed. To see $\sup_n e_nax' = ax'$ and $\sup_n x'ae_n = x'a$ $(a\in\cal A_+,\, x'\in\frak X'_+)$, simply note that $(e_n)$ is a norm-bounded a.i.~for $\cal A^r(X)$, so $\lim_n \|e_nax'-ax'\| = 0 = \lim_n \|x'ae_n-x'a\|$ $(a\in\cal A,\, x'\in\frak X')$. As for the other hypotheses of Theorem~\ref{atomicmain}, it is easy to see that they are all satisfied, so the desired conclusion follows.
	\end{proof}

	\section{The non-atomic case}
	
	As announced in the introduction, in the last section of this note, we extend our previous results to algebras of regular operators defined on non-atomic Banach lattices. Before we can state the main result of the section (Theorem~\ref{nonatomicmain} below), we recall some notions, first introduced in~\cite{Bl2}. 
	
	Following \cite{Bl2}, if $X$ is a Banach lattice and $\Pi\subset\cal L^r(X)$ is a bounded family of positive projections such that $\sigma(X)$ is a Riesz subspace of $X$ for every $\sigma\in\Pi$ and $\overline{\bigcup_{\sigma\in\Pi} \sigma(X)} = X$, then we shall call $\Pi$ a {\bf bounded generating system for $X$}. If $\{\sigma(X) : \sigma\in\Pi\}$, partially ordered by inclusion, is a directed set, we shall say $\Pi$ is {\bf directed}, and think of it as a directed set with respect to the ordering: $\sigma_1\curlyeqprec\sigma_2$ $\iff$ $\sigma_1(X)\subseteq \sigma_2(X)$ $(\sigma_1,\sigma_2\in\Pi)$. We shall write $p_\sigma$ for $\sigma|^{\sigma(X)}$ ($:=$ corestriction of $\sigma$ to $\sigma(X)$) and $\imath_{\sigma(X)}$ for the natural inclusion of $\sigma(X)$ into $X$ $(\sigma\in\Pi)$. Note that $\imath_\sigma p_\sigma = \sigma$ and $p_\sigma\imath_\sigma = \mathrm{id}_{\sigma(X)}$ $(\sigma\in\Pi)$. The class of Banach lattices with directed bounded generating systems includes $L^p$-spaces, rearrangement invariant spaces on $[0,\infty)$ (see \cite[Section~3]{Bl2}), and of course, 1-unconditional direct sums of sequences of Banach lattices of any of these kinds. 
	
	We shall write $\frak F_X^{\star,\mu}$ for the collection of all vector sublattices of $X$, isometrically lattice isomorphic to Banach lattices of the form $(\bigoplus_{i=1}^{\frak m} F_i)_\mathbf{e}$, where $(F_i)_{i=1}^{\frak m}$ is a sequence of infinite dimensional separable purely atomic Banach lattices satisfying ($\star$) for the same constant $\mu$ and $\mathbf{e}$ is a 1-unconditional basis with $|\mathbf{e}| = {\frak m}$.
	
	Given a Banach lattice $X$, a closed Riesz algebra ideal $\cal A$ of $\cal L^r(X)$ and a positive projection $\sigma$ onto a Riesz subspace of $X$, we shall denote by $\cal A_\sigma$ the closed Riesz algebra ideal of $\cal L^r(\sigma(X))$ given by 
	\[
	\{a\in\cal L^r(\sigma(X)) : \imath_\sigma ap_\sigma\in\cal A\}. 
	\]
	If $\frak X$ is an $\cal A$-module, we shall denote by $\frak X_\sigma$ the Banach lattice $\cal A_\sigma$-module whose underlying space is $\frak X$ itself and whose module products are given by 
	\[
	\cal A_\sigma\times\frak X\to\frak X,\; (a,x)\mapsto (\imath_\sigma ap_\sigma) x 
	\quad\text{and}\quad 
	\frak X\times\cal A_\sigma\to\frak X,\; (x,a)\mapsto x(\imath_\sigma ap_\sigma). 
	\]
	Clearly, if $\frak X$ is a Riesz $\cal A$-module, then $\frak X_\sigma$ is a Riesz $\cal A_\sigma$-module. To simplify notations, we shall denote the left and right module products of an element $x\in\frak X_\sigma$ with an element $a\in\cal A_\sigma$, by $ax$ and $xa$, respectively. Note that $\|ax\|\leq \|\sigma\|\|a\|\|x\|$ and $\|xa\|\leq \|\sigma\|\|a\|\|x\|$ $(a\in\cal A_\sigma,\, x\in\frak X_\sigma)$, so $\|x'a\|\leq \|\sigma\|\|a\|\|x'\|$ and $\|ax'\|\leq \|\sigma\|\|a\|\|x'\|$ $(a\in\cal A_\sigma,\, x'\in (\frak X_\sigma)')$.
	
	The main result of the section can now be stated as follows.
	
	\begin{thm}\label{nonatomicmain}
		Let $X$ be a Banach lattice with a directed bounded generating system $\Pi$ such that $\{\sigma(X) : \sigma\in\Pi\}\subseteq\frak F_X^{\star,\mu}$ for some $\mu\geq 1$. For each $\sigma\in\Pi$, let $(x_i^{(\sigma)})$ be a basis for $\sigma(X)$ satisfying $(\star)$ with constant $\mu$ in each summand, and let $(e_n^{(\sigma)})$ be the corresponding sequence of natural projections. Let $\cal A$ be a closed Riesz algebra ideal of $\cal L^r(X)$ and let $\frak X$ be a Banach lattice $\cal A$-module. Suppose, for each $\sigma\in\Pi$, \vspace{1mm}
		\begin{itemize}
			\item[--] $(x_i^{(\sigma)})$ is shrinking; 
			
			\item[--] $\cal A_\sigma\cal A_\sigma = \cal A_\sigma$;
						
			\item[--] the maps defined by $E_\sigma x' = \sup_n e_n^{(\sigma)}x'$ and $F_\sigma x' = \sup_n x'e_n^{(\sigma)}$ $(x'\in\frak X'_+)$, have weak*-closed ranges; 
			
			\item[--] $\sup_n e_n^{(\sigma)}ax' = ax'$ and $\sup_n x'ae_n^{(\sigma)} = x'a$ $(a\in (\cal A_\sigma)_+,\, x'\in (\frak X'_\sigma)_+)$;
		\end{itemize}		
		and suppose at least one of the following holds: 
		\begin{itemize}
			\item[$\bullet$] for each $\sigma\in\Pi$, the pair $\cal A_\sigma$, $\frak X_\sigma$ satisfies either $(i)$ or $(ii)$ of Theorem~\ref{atomicmain}; or 
			
			\item[$\bullet$] every bounded derivation from $\cal A^r(X)$ to $\frak X'$ is inner.  	
		\end{itemize} \vspace{1mm}
		Then every bounded derivation $D:\cal A\to\frak X'$, such that $w^*\text{-}\lim_\sigma D(\sigma a\sigma) = D(a)$ $(a\in\cal A)$, is regular.
	\end{thm}
	
	\begin{rem}
		Regarding the assumption that $\{\sigma(X) : \sigma\in\Pi\}\subseteq \frak F_X^{\star,\mu}$, all examples of Banach lattices, given at the end of the second paragraph of this section, have directed bounded generating systems satisfying this condition (see \cite[Section~3]{Bl2}). 
	\end{rem}
	
	\begin{proof}[Proof of the theorem]
		Let $D:\cal A\to\frak X'$ be a bounded derivation as in the hypotheses of the theorem. We first construct, for each $a\in\cal A_+$, a norm-bounded net $(\xi_\sigma(a))\subset\frak X'$ such that $\sup_{t\in\cal A : |t|\leq a} |D(\sigma t\sigma)|\leq \xi_\sigma(a)$. To this end, fix $\sigma\in\Pi$ and let $D_\sigma:\cal A_\sigma\to(\frak X_\sigma)'$, $a\mapsto D(\imath_\sigma ap_\sigma)$. Clearly, $D_\sigma$ is a continuous derivation. 
		
		The proof of Theorem~\ref{atomicmain} yields a sequence $(p_i^{(\sigma)})$ of mutually orthogonal o-minimal idempotents in $\cal A_\sigma$ satisfying the hypotheses of Theorem~\ref{32like}. Let $\pi_n^{(\sigma)} = \sum_{i=1}^n p_i^{(\sigma)}$ $(n\in\Bbb N)$. It is apparent from the proof of Theorem~\ref{atomicmain} that the sequence $(p_i^{(\sigma)})$ can be chosen so that
		\[
		\max\bigg\{\bigg\|\sup_n \sum_{i=1}^n |p_i^{(\sigma)} D_\sigma(p_i^{(\sigma)}) \pi_n^{(\sigma)}|\bigg\|, 
		\bigg\|\sup_n \sum_{i=1}^n |\pi_n^{(\sigma)} D_\sigma(p_i^{(\sigma)}) p_i^{(\sigma)}|\bigg\|\bigg\}\leq
		K, 
		\]
		for some constant $K$ independent of $\sigma$. (Indeed, if the pair $\cal A_\sigma$, $\frak X_\sigma$ satisfies either (i) or (ii) of Theorem~\ref{atomicmain}, then it should be clear from the proof of the same theorem that for any $K>0$ one can choose $(p_i^{(\sigma)})$ so that $\sum_i \sum_j \|p_i^{(\sigma)} D_\sigma(p_i^{(\sigma)}) p_j^{(\sigma)}\|\leq K$; while if $D(a) = ax'_D-x'_Da$ $(a\in\cal A^r(X))$ for some $x_D'\in\frak X'$, then $D_\sigma(a) = ax'_D-x'_Da$ $(a\in\cal A^r(\sigma(X)))$, so $\sup_n \|\pi_n^{(\sigma)} x'_D\pi_n^{(\sigma)}\|\leq \|\sigma\|^2\|x'_D\|$ and one can take as $K$ this last value.) It is also clear from the proof of Theorem~\ref{atomicmain} that the sequences $(u_n^{(\sigma)})$ and $(v_n^{(\sigma)})$, of the factorization of $(e_n^{(\sigma)})$ through $(\pi_n^{(\sigma)})$, can be chosen so that
		\[
		\sup_n \|u_n^{(\sigma)}\|\leq \mu
		\quad\text{ and }\quad
		\sup_n \|v_n^{(\sigma)}\|\leq \mu.
		\]
		
		Let $\tau_\sigma$ be the topology on $\cal A_\sigma$, generated by the seminorms $a\mapsto\|ab\|+\|ba\|$, $\cal A_\sigma\to\Bbb R_+\cup\{0\}$ $(b\in\cal F(\sigma(X)))$. To simplify notations, in what follows we write $\lim$ for $\tau_\sigma$-limit. From the proof of Theorem~\ref{32like} (see (\ref{Formula})), we know that
		\[
		D_\sigma(c) = \widetilde{x}_\sigma \lim_n \Psi_\sigma(c)v_n^{(\sigma)} +
		(\Theta_\sigma\circ \Phi_\sigma)(c) +
		\Big(\lim_m u_m^{(\sigma)}\Psi_\sigma(c)\Big) \widetilde{y}_\sigma
		\quad (c\in\cal A_\sigma),
		\]
		where $\widetilde{x}_\sigma$ and $\widetilde{y}_\sigma$ are weak*-limit points in $\frak X'$ of the sets $\{D_\sigma(u_n^{(\sigma)}) : n\in\Bbb N\}$ and $\{D_\sigma(v_n^{(\sigma)}) : n\in\Bbb N\}$, respectively; $\Phi_\sigma:\cal A_\sigma\to\cal A_\sigma$, $c\mapsto \lim_i \lim_j v_i^{(\sigma)}cu_j^{(\sigma)}$; $\Psi_\sigma:\cal A_\sigma\to\cal A_\sigma$, $c\mapsto \lim_r \lim_s \pi_{k_r}^{(\sigma)} (v_r^{(\sigma)}cu_s^{(\sigma)}) \pi_{k_s}^{(\sigma)}$; and
		\[
		\Theta_\sigma:\cal A_\sigma\to(\frak X_\sigma)',\, c\mapsto
		w^*\text{-}\lim_\alpha w^*\text{-}\lim_\beta \;u_{m_\alpha}^{(\sigma)} \Big(\lim_r \lim_s D_\sigma(\pi_{k_r}^{(\sigma)}c\pi_{k_s}^{(\sigma)})\Big) v_{n_\beta}^{(\sigma)},
		\]
		where $(u_{m_\alpha}^{(\sigma)})$ and $(v_{n_\beta}^{(\sigma)})$ are subnets of $(u_m^{(\sigma)})$ and $(v_n^{(\sigma)})$, respectively, such that $(u_{m_\alpha}^{(\sigma)}x')$ and $(x'v_{n_\beta}^{(\sigma)})$ weak*-converge in $\frak X'$ for every $x'\in\frak X'$.
		Thus, if $|t|\leq a\in\cal A_+$, then $|D(\sigma t\sigma)| = |D_\sigma(p_\sigma t\imath_\sigma)|\leq |D_\sigma|(p_\sigma a\imath_\sigma)$ and
		\begin{equation}\label{17}
			\begin{split}
				|D_\sigma|(p_\sigma a\imath_\sigma)&\leq
				|\widetilde{x}_\sigma| \lim_n \Psi_\sigma(p_\sigma a\imath_\sigma)v_n^{(\sigma)} \\
				&\hspace{2cm} + |\Theta_\sigma|(\Phi_\sigma(p_\sigma a\imath_\sigma)) +
				\Big(\lim_m u_m^{(\sigma)}\Psi_\sigma(p_\sigma a\imath_\sigma)\Big) |\widetilde{y}_\sigma|,
			\end{split}
		\end{equation}
		with
		\[
		|\Theta_\sigma|(c)\leq w^*\text{-}\lim_\alpha w^*\text{-}\lim_\beta\;
		\xi_{m_\alpha n_\beta}^{(\sigma)}(c) \quad (c\in\cal A_\sigma),
		\]
		where
		\[
		\begin{split}
			\xi_{mn}^{(\sigma)}(c) &= u_m^{(\sigma)} (\xi^{(\sigma)}+\eta^{(\sigma)}) cv_n^{(\sigma)} \\
			&\hspace{1cm} + u_m^{(\sigma)}
			\Big|D_\sigma\Big(\lim_k \lim_l \pi_k^{(\sigma)} c\pi_l^{(\sigma)}\Big)\Big| v_n^{(\sigma)} + u_m^{(\sigma)}c (\xi^{(\sigma)}+\eta^{(\sigma)}) v_n^{(\sigma)}
			\quad (m,n\in\Bbb N),
		\end{split}
		\]
		\[
		\xi^{(\sigma)} = \sup_n \sum_{i=1}^n |\pi_n^{(\sigma)} D_\sigma(p_i^{(\sigma)}) p_i^{(\sigma)}|
		\quad\text{and}\quad
		\eta^{(\sigma)} = \sup_n \sum_{i=1}^n |p_i^{(\sigma)} D_\sigma(p_i^{(\sigma)}) \pi_n^{(\sigma)}|
		\]
		(see the proof of Theorem~\ref{32like}).
		
		Note that if $(c_n)\subset\cal A_\sigma$ is $\tau_\sigma$-convergent, with $\tau_\sigma$-limit $c$, then $\|c\|\leq \sup_n \|c_n\|$, for $\|\lambda\|\|(c-c_n)(x)\|\leq \tau_{\lambda\otimes x}(c-c_n)$ $(x\in\sigma(X)_{[1]},\, \lambda\in\sigma(X)'_{[1]})$. Applying this fact in an iterative fashion, we find that for every $c\in\cal A_\sigma$,
		\[
		\big\|\Phi_\sigma(c)\big\|\leq
		\sup_{i,j} \big\|v_i^{(\sigma)}cu_j^{(\sigma)}\big\|\leq \mu^2\|c\|, 
		\]
		\[
		\big\|\Psi_\sigma(c)\big\|\leq
		\sup_{r,s} \big\|\pi_{k_r}^{(\sigma)}(v_r^{(\sigma)}cu_s^{(\sigma)})\pi_{k_s}^{(\sigma)}\big\|\leq \mu^2\|c\|,
		\]
		\[
		\Big\|\lim_n \Psi_\sigma(c)v_n^{(\sigma)}\Big\|
		\leq \sup_n \big\|\Psi_\sigma(c)v_n^{(\sigma)}\big\|\leq \mu^3\|c\|,
		\]
		and
		\[
		\Big\|\lim_m u_m^{(\sigma)}\Psi_\sigma(c)\Big\|
		\leq \sup_m \big\|u_m^{(\sigma)}\Psi_\sigma(c)\big\|\leq \mu^3\|c\|.
		\]
		As for $|\Theta_\sigma|$, recall that if $(x_\alpha')$ is a net in $\frak X'$ with weak*-limit $x'$ then $\|x'\|\leq \sup_\alpha \|x_\alpha'\|$, so 
		\[
		\begin{split}
			\big\||\Theta_\sigma|(c)\big\| &\leq
			\sup_{\alpha,\beta} \big\|\xi_{m_\alpha n_\beta}^{(\sigma)}(c)\big\| \\
			&\leq 2\mu^2 \|\sigma\|^2\|\xi^{(\sigma)}+\eta^{(\sigma)}\|\|c\| +
			\mu^2 \|\sigma\|^2\|D_\sigma\|\sup_{k,l} \|\pi_k^{(\sigma)}c\pi_l^{(\sigma)}\| \\
			&\leq \mu^2 \|\sigma\|^2\big(4K + \|D\|\|\sigma\|\big)\|c\| \quad (c\in\cal A_\sigma).
		\end{split}
		\]
		Clearly, 
		\[
		\|\widetilde{x}_\sigma\|\leq \sup_n \|D_\sigma(u_n^{(\sigma)})\|\leq \mu\|\sigma\|\|D\|
		\quad\text{and}\quad
		\|\widetilde{y}_\sigma\|\leq \sup_n \|D_\sigma(v_n^{(\sigma)})\|\leq \mu\|\sigma\|\|D\|.
		\]
		Combining all the above norm-estimates with (\ref{17}), one obtains that
		\[
		\begin{split}
			\big\||D_\sigma|(p_\sigma a\imath_\sigma)\big\|
			&\leq 2\mu^4\|\sigma\|\|D\|\|p_\sigma a\imath_\sigma\| \\
			&\hspace{2cm} + \mu^2\|\sigma\|^2\big(4K + \|D\|\|\sigma\|\big)
			\|\Phi_\sigma(p_\sigma a\imath_\sigma)\| \\
			&\leq \big[2\mu^4\|\sigma\|^2\|D\| + \mu^4\|\sigma\|^3(4K + \|D\|\|\sigma\|)\big]\|a\| \quad (\sigma\in\Pi).
		\end{split}
		\]
		We let $\xi_\sigma(a) := |D_\sigma|(p_\sigma a\imath_\sigma)$ $(\sigma\in\Pi)$. 
		
		Now, since the net $(\xi_\sigma(a))_\sigma$ is norm-bounded, it must contain a weak*-convergent subnet, $(\xi_{\sigma_\alpha}(a))_\alpha$ say. Let $\xi(a) := w^*\text{-}\lim_\alpha \xi_{\sigma_\alpha}(a)$. Then, for every $t\in\cal A$ such that $|t|\leq a$, combining the fact that $D(t) = w^*\text{-}\lim_\sigma D(\sigma t\sigma)$, with the weak*-closedness of $\frak X'_+$ and the inequalities $|D(\sigma t\sigma)|\leq \xi_\sigma(a)$ $(\sigma\in\Pi)$, we obtain that
		\[
		|D(t)|\leq w^*\text{-}\lim_\alpha \xi_{\sigma_{\kappa_\alpha}}(a) = \xi(a). 
		\]
		As $a\in\cal A_+$ was arbitrary, the regularity of $D$ follows readily from this. 
	\end{proof}
	
	As announced in the introduction, we conclude the section -- and the note -- by discussing some natural situations in which the hypotheses of Theorem~\ref{nonatomicmain} hold. 
	
	\begin{cor}\label{C1}
		Let $X$, $\Pi$, $(x_i^{(\sigma)})$ and $(e_n^{(\sigma)})$ $(\sigma\in\Pi)$ be as in Theorem~\ref{nonatomicmain}, and suppose in addition, that $\lim_\sigma \|\lambda-\lambda\sigma\| = 0$ $(\lambda\in X')$. Let $e := \mathrm{id}_X$ and let $\frak X$ be a Banach lattice $\cal L^r(X)$-module such that, 
		\begin{itemize}
			\item[--] $w^*\text{-}\lim_\sigma \sigma x' = ex'$ and $w^*\text{-}\lim_\sigma x'\sigma = x'e$ $(x'\in\frak X')$; and 
			
			\item[--] $w^*\text{-}\lim_n e_n^{(\sigma)}x' = \sigma x'$ and $w^*\text{-}\lim_n x'e_n^{(\sigma)} = x'\sigma$ $(x'\in\frak X',\, \sigma\in\Pi)$. 
		\end{itemize} 
		If in addition any of the following holds, then every bounded derivation $D:\cal L^r(X)\to\frak X'$ is regular:
		\begin{itemize} 
			\item[i)] $X$ contains no copies of $\ell_1$ and $\frak X$ contains no complemented copies of $\ell_1$;
			
			\item[ii)] $X$ is reflexive and $\dim{e_n^{(\sigma)}\frak X'e_n^{(\sigma)}}<\infty$ $(n\in\Bbb N,\,\sigma\in\Pi)$; or
			
			\item[iii)] $X$ contains no copies of $\ell_1$ and every bounded derivation from $\cal A^r(X)$ to $\frak X'$ is inner.
		\end{itemize}
	\end{cor}
	
	\begin{proof}
		To simplify notations, we set $\cal A := \cal L^r(X)$. Let $\frak X$ be as in the hypotheses. We first notice that every pair $\cal A_\sigma$, $\frak X_\sigma$ satisfies the requirements of Theorem~\ref{nonatomicmain}. Indeed, for every $\sigma\in\Pi$, one has $\cal A_\sigma = \cal L^r(\sigma(X))$ so $\cal A_\sigma\cal A_\sigma = \cal A_\sigma$; $(x_i^{(\sigma)})$ is shrinking,  for $\sigma(X)$ contains no copies of $\ell_1$ (see \cite[Theorem~1.c.9]{LT}); $E_\sigma(\frak X')$ and $F_\sigma(\frak X')$ are weak*-closed, for $E_\sigma x' = w^*\text{-}\lim_n e_n^{(\sigma)}x' = \sigma x'$ and $F_\sigma x' =  w^*\text{-}\lim_n x'e_n^{(\sigma)} = x'\sigma$ $(x'\in\frak X')$, and multiplication by elements of $\cal A$ is weak*-weak*-continuous on $\frak X'$; also $\sup_n e_n^{(\sigma)}ax' = ax'$ and $\sup_n x'ae_n^{(\sigma)} = x'a$ $(a\in (\cal A_\sigma)_+,\,x'\in\frak X'_+)$, for $\sup_n e_n^{(\sigma)}x' = w^*\text{-}\lim_n e_n^{(\sigma)}x' = \sigma x'$ and $\sup_n x'e_n^{(\sigma)} = w^*\text{-}\lim_n x'e_n^{(\sigma)} = x'\sigma$ $(x'\in\frak X'_+)$; and lastly, the absence of complemented copies of $\ell_1$ in $\frak X$ (i.e., the second requirement of (i) above) would imply that (i) from Theorem~\ref{atomicmain} holds, while $X$ reflexive and $\dim{e_n^{(\sigma)}\frak X'e_n^{(\sigma)}}<\infty$ $(n\in\Bbb N,\,\sigma\in\Pi)$ (i.e., (ii) above) would imply that (ii) from Theorem~\ref{atomicmain} holds. Trivially, if (iii) above holds, then so does the condition under the second bullet point of Theorem~\ref{nonatomicmain}. 
		
		Now let $D:\cal A\to\frak X'$ be a bounded derivation. To finish our proof, we will show that $w^*\text{-}\lim_\sigma D(\sigma a\sigma) = D(a)$ $(a\in\cal A)$. To this end, let $\omega$ be the topology on $\frak X'$ generated by the seminorms $\omega_{\mathbf{b},\mathbf{f}}$ $(\mathbf{b}\in\cal B\times\cal B,\, \mathbf{f}\in \frak X\times (\frak X\cap (\frak X'e)^\perp)\times (\frak X\cap (e\frak X')^\perp))$, where $\omega_{\mathbf{b},\mathbf{f}}:\frak X'\to\Bbb R_+\cup\{0\}$, $x'\mapsto |f_1(b_1x'b_2)| + |f_2(b_1x')| + |f_3(x'b_2)|$. Clearly, $\omega$ is coarser than the weak*-topology on $e\frak X'+\frak X'e =: \frak Y$. Furthermore, $\frak Y$ with the topology induced by $\omega$ is Hausdorff. The proof of this goes along the same lines as in the proof of Theorem~\ref{atomicmain}. Simply note that the maps $x'\mapsto ex'$, $\frak X'\to\frak X'$ and $x'\mapsto x'e$, $\frak X'\to\frak X'$ form a commuting pair of orthogonal projections, and that for every $x\in\frak X$ and $x'\in\frak X'$, $(ex')(x) = \lim_\sigma \lim_n (e_n^{(\sigma)}x')(x)$ and $(x'e)(x) = \lim_\sigma \lim_n (x'e_n^{(\sigma)})(x)$.
		
		To achieve the desired conclusion, it will suffice to show that $\omega\text{-}\lim_\sigma \sigma D(a\sigma) = eD(a)$ and $\omega\text{-}\lim_\sigma D(\sigma) a\sigma = D(e)a$ on $(\frak Y,\omega)$, for then, since $e\frak X'$ and $\frak X'e$ are weak*-closed and both nets are norm-bounded, we would have that 
		\[
		\begin{split}
			D(a) &= \omega\text{-}\lim_\sigma \sigma D(a\sigma) + \omega\text{-}\lim_\sigma D(\sigma) a\sigma \\
			&= w^*\text{-}\lim_\sigma \sigma D(a\sigma) + w^*\text{-}\lim_\sigma D(\sigma) a\sigma = w^*\text{-}\lim_\sigma D(\sigma a\sigma).	
		\end{split}
		\] 
		Since (by hypothesis) $w^*\text{-}\lim_\sigma \sigma D(a) = eD(a)$, to show $\omega\text{-}\lim_\sigma \sigma D(a\sigma) = eD(a)$, it will suffice to show $\lim_\sigma \omega_{\mathbf{b},\mathbf{f}}(\sigma D(a\sigma-a)) = 0$ for every seminorm in the generating system for $\omega$. In turn, the latter follows on noting that for any $\omega_{\mathbf{b},\mathbf{f}}$,   
		\[
		\begin{split}
			\omega_{\mathbf{b},\mathbf{f}}(\sigma D(a\sigma-a)) &= |f_1(b_1\sigma D(a\sigma-a)b_2)| + |f_2(b_1\sigma D(a\sigma-a))| \\
			&= |f_1(D(b_1\sigma(a\sigma-a))b_2) - f_1(D(b_1\sigma)(a\sigma-a)b_2)| \\
			&\hspace{6cm} + |f_2(D(b_1\sigma(a\sigma-a)))| \\
			&\leq \|f_1\|\|D\|(\|b_1\sigma(a\sigma-a)\|\|b_2\|+\|b_1\sigma\|\|(a\sigma-a)b_2\|) \\
			&\hspace{6cm} + \|f_2\|\|D\|\|b_1\sigma(a\sigma-a)\|,
		\end{split}
		\]
		and as $b_1,b_2\in\cal F(X)$, both $\|b_1\sigma(a\sigma-a)\|$ ($ = \|(b_1\sigma-b_1)a\sigma + ((b_1a)\sigma-(b_1a)) + (b_1-b_1\sigma)a\|$) and $\|(a\sigma-a)b_2\|$ ($ = \|a(\sigma b_2-b_2)\|$) tend to $0$. As for $\omega\text{-}\lim_\sigma D(\sigma) a\sigma = D(e)a$, notice that for every seminorm $\omega_{\mathbf{b},\mathbf{f}}$,
		\[
		\begin{split}
			\omega_{\mathbf{b},\mathbf{f}}(D(\sigma-e)a) &= |f_1(b_1D(\sigma-e)ab_2)| + |f_3(D(\sigma-e)ab_2)| \\
			&= |f_1(b_1D((\sigma-e)ab_2)) - f_1(b_1(\sigma-e)D(ab_2))| + |f_3(D(\sigma-e)ab_2)| \\
			&= \|f_1\|\|D\|(\|b_1\|\|(\sigma-e)ab_2\|+\|b_1(\sigma-e)\|\|ab_2\|) \\
			&\hspace{7cm} + \|f_3\|\|D\|\|(\sigma-e)ab_2\|, 
		\end{split}	
		\]
		and 
		\[
		\begin{split}
			\omega_{\mathbf{b},\mathbf{f}}(D(\sigma)(a\sigma-a)) &= |f_1(b_1D(\sigma)(a\sigma-a)b_2)| + |f_3(D(\sigma)(a\sigma-a)b_2)| \\
			&\leq (\|f_1\|\|b_1\|\|D\|\|\sigma\| + \|f_3\|\|D\|\|\sigma\|)\|(a\sigma-a)b_2\|, 
		\end{split}
		\]
		so $\lim_\sigma \omega_{\mathbf{b},\mathbf{f}}(D(\sigma-e)a) = 0 = \lim_\sigma \omega_{\mathbf{b},\mathbf{f}}(D(\sigma)(a\sigma-a))$, for $\|(\sigma-e)ab_2\|$, $\|b_1(\sigma-e)\|$ and $\|(a\sigma-a)b_2\|$, all tend to $0$ with $\sigma$. It follows from this that $\lim_\sigma \omega_{\mathbf{b},\mathbf{f}}(D(\sigma)a\sigma-D(e)a) = 0$, and hence, the desired result.  
	\end{proof}
	
	\begin{cor}
		Let $X$, $\Pi$, $(x_i^{(\sigma)})$ and $(e_n^{(\sigma)})$ $(\sigma\in\Pi)$ be as in Theorem~\ref{nonatomicmain}, and suppose in addition that $\lim_\sigma \|\lambda-\lambda\sigma\| = 0$ $(\lambda\in X')$. Let $\frak X$ be a Banach lattice $\cal A^r(X)$-module in which $(e_n^{(\sigma)}x)$ and $(xe_n^{(\sigma)})$ converge for every $x\in\frak X$ and $\sigma\in\Pi$. If in addition any of the following holds, then every bounded derivation $D:\cal A^r(X)\to\frak X'$ is regular:
		\begin{itemize}
			\item[i)] $X$ contains no copies of $\ell_1$ and $\frak X$ contains no complemented copies of $\ell_1$; or
			
			\item[ii)] $X$ is reflexive and $\dim{e_n^{(\sigma)}\frak X'e_n^{(\sigma)}}<\infty$ $(n\in\Bbb N,\,\sigma\in\Pi)$.
		\end{itemize}
	\end{cor}
	
	\begin{proof}
		To simplify notations, let $\cal A := \cal A^r(X)$. Note first $\cal A_\sigma = \cal A^r(\sigma(X))$ $(\sigma\in\Pi)$. Like in the previous corollary, the absence of copies of $\ell_1$ in $X$ implies $(x_i^{(\sigma)})$ is shrinking $(\sigma\in\Pi)$. It follows that $(e_n^{(\sigma)})$ is a norm-bounded approximate identity for $\cal A_\sigma$, and hence, that $\cal A_\sigma\cal A_\sigma = \cal A_\sigma$. Also from the fact that $(e_n^{(\sigma)})$ is an a.i.~for $\cal A_\sigma$, one readily obtains that $\sup_n e_n^{(\sigma)}ax' = ax'$ and $\sup_n x'ae_n^{(\sigma)} = x'a$ $(a\in (\cal A_\sigma)_+,\,x'\in\frak X')$. That for every $\sigma$ the ranges of $E_\sigma$ and $F_\sigma$ are weak*-closed follows from the fact that they are both dual projections (see the proof of Corollary~\ref{cor3}). That the pair $\cal A_\sigma$, $\frak X_\sigma$ satisfies either (i) or (ii) of Theorem~\ref{nonatomicmain}, is guarantee, like in the previous corollary, by conditions (i) and (ii) of the hypotheses. 
		
		Lastly, note that $\lim_\sigma \|\sigma a\sigma - a\| = 0$ $(a\in\cal A^r(X))$ (indeed, recall $\lim_\sigma \|x-\sigma x\| = 0$ $(x\in X)$, $\lim_\sigma \|\lambda-\lambda\sigma\| = 0$ $(\lambda\in X')$, and $\cal F(X)$ is norm-dense in $\cal A^r(X)$), so if $D:\cal A^r(X)\to\frak X'$ is a bounded derivation, then $D(a) = D(\lim_\sigma \sigma a\sigma) = \lim_\sigma D(\sigma a\sigma) = w^*\text{-}\lim_\sigma D(\sigma a\sigma)$ $(a\in\cal A^r(X))$. One can then apply Theorem~\ref{nonatomicmain}.
	\end{proof}

\end{document}